\documentclass[10pt,conference]{IEEEtran}
\usepackage[utf8]{inputenc}
\usepackage[english]{babel}
\usepackage{amsmath}
\usepackage{censor}
\usepackage{amssymb}
\usepackage{multicol}
\usepackage{enumitem}
\usepackage{amsfonts}
\usepackage{amsthm}
\usepackage{graphicx}
\usepackage{stmaryrd}
\usepackage{mathtools}
\usepackage{hyperref}

\DeclarePairedDelimiter{\ceil}{\lceil}{\rceil}
\DeclarePairedDelimiter{\floor}{\lfloor}{\rfloor}

\DeclarePairedDelimiter{\norm}{\lVert}{\rVert}

\newcommand{\R}{\mathbb{R}}
\newcommand{\CC}{\mathbb{C}}
\newcommand{\N}{\mathbb{N}}
\newcommand{\cl}{\overline}
\newcommand\Define[1]{\textbf{#1}}
\DeclareMathOperator{\supp}{supp}

\newcounter{counter}
\theoremstyle{definition}
\newtheorem{theorem}[counter]{Theorem}
\newtheorem*{theorem*}{Theorem}
\newtheorem{proposition}[counter]{Proposition}
\newtheorem{lemma}[counter]{Lemma}
\newtheorem{definition}[counter]{Definition}
\newtheorem{corollary}[counter]{Corollary}
\newtheorem{example}[counter]{Example}
\newtheorem{remark}[counter]{Remark}

\newcommand{\bigovee}{\mathop{\vphantom{\sum}\mathchoice%
        {\vcenter{\hbox{\huge $\ovee$}}}%
        {\vcenter{\hbox{\Large $\ovee$}}}%
        {\ovee}{\ovee}}\displaylimits}

\begin{document}
\StopCensoring
\title{A characterisation of ordered abstract probabilities}

\author{\IEEEauthorblockN{\censor{Abraham Westerbaan}}
\IEEEauthorblockA{\censor{Radboud Universiteit Nijmegen}\\
                \censor{bram@westerbaan.name}}
\and\IEEEauthorblockN{\censor{Bas Westerbaan}}
\IEEEauthorblockA{\censor{University College London}\\
                \censor{bas@westerbaan.name}}
\and\IEEEauthorblockN{\censor{John van de Wetering}}
\IEEEauthorblockA{\censor{Radboud Universiteit Nijmegen}\\
                \censor{john@vdwetering.name}}}

\maketitle

\begin{abstract}
In computer science,
    especially when dealing with quantum computing
        or other non-standard models of computation,
    basic notions in probability theory
    like ``a predicate'' vary wildly.
There seems to be one constant:
    the only useful example of an algebra of probabilities is
    the real unit interval.
In this paper we try to explain this phenomenon.
We will show that the structure of the real unit interval naturally
arises from a few reasonable assumptions.
We do this by studying \emph{effect monoids}, an abstraction of the algebraic structure of the real unit interval:
it has an addition $x+y$ which is only defined when $x+y\leq 1$ and an involution $x\mapsto 1-x$ which make it an effect algebra, in combination with an associative (possibly non-commutative) multiplication.
Examples include the unit intervals of ordered rings and Boolean algebras.

We present a structure theory for effect monoids that are
$\omega$-complete, i.e.\ where every increasing sequence has a supremum.

We show that any $\omega$-complete effect monoid 
embeds into the direct sum of a Boolean algebra and
the unit interval of a commutative unital
    C$^*$-algebra.
This gives us from first principles a dichotomy between sharp logic, represented by the Boolean algebra part of the effect monoid, and probabilistic logic, represented by the commutative C$^*$-algebra.
Some consequences of this characterisation are that the multiplication must always be commutative, and that the unique $\omega$-complete effect monoid without zero divisors and more than 2 elements must be the real unit interval.
Our results give an algebraic characterisation and motivation for why any physical or logical theory would represent probabilities by real numbers.
\end{abstract}


\section{Introduction}


Probability theory in the quantum realm is different in important ways from that of the classical world. Nevertheless, they both crucially rely on real numbers to represent probabilities of events. This makes sense as observations of quantum systems must still be interpreted trough classical means. However, in principle one can imagine a world governed by different physical laws where even the standard notion of a probability is different, or wish to study probabilistic models where one does not care about the specifics of their probabilities; such an approach can for instance be found in categorical quantum mechanics~\cite{abramsky2004categorical,gogioso2017fantastic,tull2016reconstruction}.
In this paper we study a reasonable class of alternatives to the real unit interval as the set of allowed probabilities. We will establish that this quite general seeming class actually only contains (continuous products of) the real unit interval. This shows that any `reasonable' enough physical theory must necessarily be based on probabilities represented by real numbers.

In order to determine the right set of alternatives to the real unit interval we must first find out what structure is crucial for abstract probabilities. There are a variety of operations on the real unit interval that are used in their interpretation as probabilities. First of all, to be able to talk about coarse-graining the probabilities of mutually exclusive events, we must be able to take the sum $x+y$ of two probabilities $x,y\in[0,1]$ as long as $x+y\leq 1$. Second, in order to represent the \emph{complement} of an event we require the involution given by $x^\perp \equiv 1-x$. The probability $x^\perp$ is the unique number such that $x+x^\perp = 1$. Axiomatising this structure of a partially defined addition combined with an involution defines an \emph{effect algebra}~\cite{foulis1994effect}.
The unit interval of course also has a multiplication $x\cdot y$. This operation is needed in order to talk about, for instance, joint distributions. 
An \emph{effect monoid} is an effect algebra with an associative distributive (possibly non-commutative) multiplication, and hence axiomatises these three interacting algebraic structures (addition, involution and multiplication) present in the unit interval.

In order to define an analogue to Bayes' theorem we would also need the division operation that is available in the unit interval: when $x\leq y$, then there is a probability $z$ such that $y\cdot z = x$ (namely $z=x/y$). We actually will not require the existence of such a division operation, as it turns out to follow (non-trivially) from our final requirement:

A property that sets the unit interval $[0,1]$ apart from, for instance, the rational numbers between $0$ and $1$, is that $[0,1]$ is closed under taking limits. In particular, each ascending chain of probabilities $x_1\leq x_2\leq \ldots$ has a supremum. In other words: the unit interval is $\omega$-complete.
We have then arrived at our candidate for an abstract notion of the set of probabilities: an $\omega$-complete effect monoid. 

Further motivation for the use of this structure as a natural candidate for the set of probabilities is its prevalence in \emph{effectus theory}.
This is a recent approach to categorical logic~\cite{cho2015introduction} and a general framework to deal with notions such as states, predicates, measurement and probability in deterministic, (classical) probabilistic and quantum settings~\cite{kentathesis,basthesis}.
The set of probabilities in an effectus have the structure of an effect monoid.
Examples of effectuses include any generalised probabilistic theory~\cite{barrett2007information}, where the probabilities are the unit interval, but also any topos (and in fact any extensive category with final object), where the probabilities are the Boolean values $\{0,1\}$~\cite{kentathesis}.

An effectus defines the sum of some morphisms. In a \emph{$\omega$-effectus}, this is strengthened to the existence of some countable sums (making it a \emph{partially additive category}~\cite{arbib1980partially}). In such an effectus the probabilities form an $\omega$-complete effect monoid~\cite{kentathesis}.

Effect monoids are of broader interest than only to study effectuses:
examples of effect monoids include all  Boolean algebras
    and unit intervals of partially ordered rings.
Furthermore any effect monoid can be used to define
    a generalised notion of convex set
    and convex effect algebra
    (by replacing the usual unit interval by elements of the effect monoid,
        see~\cite[179 \& 192]{basthesis}
        or~\cite{jacobs2011probabilities}).

This now raises the question of how close a probability theory based on an $\omega$-complete effect monoid is to regular probability theory.

Our main result is that $\omega$-complete effect monoids can always be embedded into a direct sum
of an $\omega$-complete Boolean algebra and the unit interval of a 
bounded-$\omega$-complete commutative unital C$^*$-algebra.
The latter is isomorphic to~$C(X)$ for some basically disconnected compact Hausdorff space~$X$.
If the effect monoid is \emph{directed complete}, 
such thay any directed set has a supremum, 
then it is even \emph{isomorphic}
to the direct sum of a complete Boolean algebra
and the unit interval of a monotone complete commutative C$^*$-algebra.

This result basically states that any $\omega$-complete effect monoid can be split up into a sharp part (the Boolean algebra), and a convex probabilistic part (the commutative C*-algebra). This then gives us from basic algebraic and order-theoretic considerations a dichotomy between sharp and fuzzy logic.

As part of the proof of this embedding theorem we find an assortment of additional structure present in $\omega$-complete effect monoids: it has a partially defined division operation, it is a lattice, and multiplication must necessarily be normal (i.e.~preserve suprema).

The classification also has some further non-trivial consequences. 
In particular, it shows that any $\omega$-complete effect monoid must necessarily be
    commutative.

Finally, we use the classification to show that an $\omega$-complete effect
    monoid without zero divisors must either be trivial, $\{0\}$,
    the two-element
    Boolean algebra, $\{0,1\}$, or the unit interval, $[0,1]$. 
This gives a new
    characterisation of the real unit interval as the unique 
     $\omega$-complete effect monoid without zero 
        divisors and more than two elements, and could be seen as a generalisation of the well-known result that the set of real numbers is the unique Dedekind-complete Archimedian ordered field.

In so far as the structure of an $\omega$-complete effect monoid is
    required for common actions involving probabilities, (coarse-graining, negations, joint distributions, limits)
    our results motivate the usage of real numbers in any hypothetical
    alternative physical theory.

\section{Preliminaries}
Before we state the main results of this paper technically,
    we recall the definitions of the structures involved.

\begin{definition}
    An \Define{effect algebra}
    (EA) $(E, \ovee, 0, (\ )^\perp)$ is a set~$E$ with
    distinguished element~$0 \in E$,
    partial binary operation~$\ovee$ (called \Define{sum})
    and (total) unary operation~$x \mapsto~x^\perp$ (called \Define{complement}),
    satisfying the following axioms,
    writing~$x\perp y$ whenever~$x \ovee y$ is defined
        and~$1 \equiv 0^\perp$.
\begin{itemize}
\item Commutativity: if $x\perp y$, then $y \perp x$ and $x\ovee y = y \ovee x$.
\item Zero: $x\perp 0$ and $x\ovee 0 = x$.
\item Associativity: if $x\perp y$ and $(x\ovee y)\perp z$, then
    $y\perp z$,
$x\perp (y \ovee z)$, and $(x\ovee y) \ovee z = x\ovee (y\ovee z)$.
\item For any~$x \in E$, the complement
    $x^\perp$ is the unique element with $x \ovee x^\perp = 1$.
        
\item If $x\perp 1$ for some~$x\in E$, then $x=0$.
\end{itemize}

For~$x,y \in E$ we write~$x \leq y$ whenever there is a~$z \in E$
    with~$x \ovee z = y$.
This turns~$E$ into a poset with minimum~$0$ and maximum~$1$.
The map~$x \mapsto x^\perp$ is an order anti-isomorphism.
Furthermore~$x \perp y$ if and only if~$x \leq y^\perp$.
If~$x \leq y$, then the element~$z$ with~$x \ovee z = y$
    is unique and is denoted by~$y \ominus x$~\cite{foulis1994effect}.

A morphism~$f\colon E \to F$ between effect algebras
    is a map such that~$f(1)=1$
        and~$f(x) \perp f(y)$ whenever~$x \perp y$,
            and then~$f(x \ovee y) = f(x) \ovee f(y)$.
    A morphism necessarily preserves the complement, $f(x^\perp) = f(x)^\perp$,
    and the order: $x\leq y \implies f(x)\leq f(y)$.
    A morphism is an \Define{embedding} when it is also \Define{order reflecting}:
    if $f(x)\leq f(y)$ then $x\leq y$.
    Observe that an embedding is automatically injective.
    We say $E$ and $F$ are \Define{isomorphic} and write $E\cong F$ when 
    there exists an isomorphism (i.e.~a bijective morphism whose inverse is a morphism too) from $E$ to $F$.
    Note that an isomorphism is the same as a surjective embedding.
\end{definition}

\begin{example}\label{ex:orthomodularlattice}
    Let $(B,0,1,\wedge, \vee, (\ )^\perp)$ be an orthomodular lattice. Then $B$ is an effect algebra with the partial addition defined by 
    $x\perp y \iff x\wedge y = 0$ and in that case  $x\ovee y = x\vee y$. The complement, $(\ )^\perp$, is given by the orthocomplement, $(\ )^\perp$.
    The lattice order coincides with the effect algebra order (defined
    above). See e.g.~\cite[Prop.~27]{basmaster}.
\end{example}

\begin{example}\label{ex:cstaralgebra}
Let~$G$ be an ordered abelian group
    (such as the self-adjoint part of a C$^*$-algebra).
    Then any interval $[0,u]_G \equiv
    \{a\in G~;~ 0\leq a \leq u\}$ where~$u$
    is a positive element of~$G$ forms an effect algebra, with
    addition given by $a\perp
    b \iff a+b\leq u$ and in that case  $a\ovee b = a+b$. The complement is
    defined by $a^\perp = u-a$.
    The effect algebra order on~$[0,u]_G$ coincides
    with the regular
    order on~$G$.  
    
    In particular,
    the set of effects $[0,1]_C$
    of a unital C$^*$-algebra~$C$ forms an effect algebra
    with $a\perp b\iff a+b\leq 1$,
    and~$a^\perp = 1-a$.
\end{example}

Effect algebras have been studied extensively
    (to name a few:
    \cite{chajda2009every,ravindran1997structure,
    gudder1996examples,gudder1996effect,jenvca2001blocks,
    dvurevcenskij2010every,foulis2010type})
    and even found surprising applications in quantum contextuality~~\cite{staton2018effect,roumen2016cohomology} and the study of Lebesque integration~\cite{jacobs2015effect}.
The following remark gives some categorical motivation to the definition of effect algebras.

\begin{remark}
  An effect algebra is a \Define{bounded poset}: a partially ordered
  set with a minimal and maximal element. In~\cite{mayet1995classes}
  it is shown that any bounded poset $P$ can be embedded into an
  orthomodular poset $K(P)$. This is known as the \Define{Kalmbach
  extension}~\cite{kalmbach1977orthomodular}. This extends to a
  functor from the category of bounded posets to the category of
  orthomodular posets, and this functor is in fact left adjoint to
  the forgetful functor going in the opposite
  direction~\cite{harding2004remarks}. This adjunction gives rise
  to the \Define{Kalmbach monad} on the category of bounded posets.
  The Eilenberg--Moore category for the Kalmbach monad is isomorphic
  to the category of effect algebras, and hence effect algebras
    are in fact algebras over bounded posets~\cite{jenvca2015effect}.
\end{remark}

The category of effect algebras is both complete and cocomplete.
There is also an algebraic tensor product of effect algebras that makes the category of effect algebras symmetric
monoidal~\cite{jacobs2012coreflections}. The monoids in the category
of effect algebras resulting from this tensor product are called
effect monoids, and they can be explicitly defined as follows:

\begin{definition}\label{def:effectmonoid}
    An \Define{effect monoid} (EM) is an effect algebra $(M,\ovee, 0, ^\perp,
    \,\cdot\,)$ with an additional (total) binary operation $\,\cdot\,$,
such that the following conditions hold 
    for all~$a,b,c \in M$.
\begin{itemize}
\item Unit: $a\cdot 1 = a = 1\cdot a$.
\item Distributivity: if $b\perp c$, then $a\cdot b\perp  a\cdot c$,\quad
    $b\cdot a\perp c\cdot a$,
        $$a\cdot (b\ovee c) \ =\  (a\cdot b) \ovee (a\cdot c),\ \text{and}\ 
(b\ovee c)\cdot a \ =\  (b\cdot a) \ovee (c\cdot a).$$
Or, in other words,
    the operation~$\,\cdot\,$ is bi-additive.
\item Associativity: $a \cdot (b\cdot c) = (a \cdot b) \cdot c$.
\end{itemize}
We call an effect monoid~$M$ \Define{commutative} if $a\cdot b = b \cdot a$
    for all~$a,b \in M$;
an element $p$ of~$M$  \Define{idempotent}
whenever~$p^2 \equiv p \cdot p = p$;
elements~$a$, $b$  of~$M$ \Define{orthogonal}
when $a\cdot b = b\cdot a = 0$;
and we denote the set of idempotents of $M$ by~$P(M)$.
\end{definition}

\begin{example}\label{ex:booleanalgebra}
Any Boolean algebra  $(B,0,1,\wedge,\vee,(\ )^\perp)$,
    being an orthomodular lattice,
    is an effect algebra by Example~\ref{ex:orthomodularlattice},
and, moreover,  a commutative effect monoid with
    multiplication defined by $x \cdot y= x \wedge y$.
    Conversely,
    any orthomodular lattice
    for which~$\wedge$ distributes over~$\ovee$ (and thus~$\vee$)
    is a Boolean algebra.
\end{example}

\begin{example}\label{ex:CX}
    The unit interval $[0,1]_R$ of any (partially)
    ordered unital ring~$R$
    (in which the sum $a+b$ and product $a\cdot b$
    of positive elements~$a$ and~$b$ are again positive)
    is an effect monoid.

    Let, for example, $X$ be a compact Hausdorff space. We denote its space of
    continuous functions into the complex numbers by
    $C(X)\equiv\{f:X\rightarrow \CC~;~f \text{ continuous}\}$. This is
    a commutative unital C$^*$-algebra
    (and conversely by the Gel'fand theorem, any
    commutative C$^*$-algebra with unit is of this form) and hence its
    unit interval $[0,1]_{C(X)} = \{f:X\rightarrow [0,1]\}$ is a
    commutative effect monoid.

    In~\cite[Ex.~4.3.9]{kentathesis} and~\cite[Cor.~51]{basmaster}
        two different non-commutative effect monoids are constructed.
\end{example}

\begin{definition}
    Let $M$ and $N$ be effect monoids. A \Define{morphism} from $M$ to $N$ is
    a morphism of effect algebras with the added condition
    that~$f(a\cdot b) = f(a)\cdot f(b)$ for all~$a,b\in M$.
    Similar to the case of effect algebras,
        an \Define{embedding}~$M \to N$ is a morphism that is order reflecting.
    Also here an \Define{isomorphism} of effect monoids
        is the same thing as a surjective embedding of effect monoids.
\end{definition}

\begin{example}\label{ex:stone-embedding}
It is well-known that any Boolean algebra $B$
    is isomorphic to the set of clopens of
    its \Define{Stone space} $X_B$. This yields an effect monoid
    embedding from $B$ into~$[0,1]_{C(X_B)}$.
\end{example}

\begin{remark}
    A physical or logical theory which has probabilities of the form $[0,1]_{C(X)}$ can be seen as a theory with a natural notion of space, where probabilities are allowed to vary continuously over the space $X$. Such a spatial theory is considered in for instance Ref.~\cite{moliner2017space}.
\end{remark}

\begin{example}
    Given two effect algebras/monoids~$E_1$ and $E_2$ we define
    their \Define{direct sum} $E_1\oplus E_2$ as the Cartesian
    product with pointwise operations. This is again an effect
    algebra/monoid.
\end{example}

\begin{example}\label{cornersexample}
    Let~$M$ be an effect monoid and let~$p \in M$ be some idempotent.
    The subset~$pM \equiv \{p \cdot a; \ a\in M\}$
        is called the \Define{left corner} by~$p$
        and is an effect monoid with~$(p\cdot a)^\perp \equiv p \cdot a^\perp$
            and all other operations inherited from~$M$.
    Later we will see that~$a \mapsto (p\cdot a, p^\perp \cdot a)$
        is an isomorphism~$M \cong p M \oplus p^\perp M$.
Analogous facts hold for the \Define{right corner}~$Mp \equiv \{a \cdot p; \ a\in M\}$.
\end{example}

\begin{definition}
    Let $E$ be an effect algebra. A \Define{directed set} $S\subseteq
    E$ is a non-empty set such that for all $a,b\in S$ there exists a $c\in
    S$ such that $a,b\leq c$. $E$ is \Define{directed complete} when
    for any directed set $S$ there is a supremum $\bigvee S$.
    It is~\Define{$\omega$-complete}
        if directed suprema of countable sets exist, or equivalently if any increasing sequence $a_1\leq a_2\leq \ldots$ in $E$ has a supremum.
\end{definition}

\begin{remark}
    A \textbf{d}irected \textbf{c}omplete \textbf{p}artially \textbf{o}rdered set is often referred to by the shorthand \Define{dcpo}. These structures lie at the basis of domain theory and are often encountered when studying denotational semantics of programming languages as they allow for a natural way to talk about fix points of recursion. Note that being $\omega$-complete is strictly weaker. For effect algebras we could have equivalently defined directed completeness with respect to downwards directed sets, as the complement is an order anti-isomorphism.
\end{remark}

\begin{example}
    Let~$B$ be a $\omega$-complete Boolean algebra. 
    Then~$B$ is a $\omega$-complete effect monoid.
    If~$B$ is complete as a Boolean algebra, then~$B$
        is directed-complete as effect monoid.
\end{example}

\begin{example}
    Let $X$ be an \Define{extremally disconnected} compact Hausdorff
    space, i.e.~where the closure of every open set is  open.
    Then $[0,1]_{C(X)}$ is a directed-complete effect monoid.
    If~$X$ is a \Define{basically disconnected} \cite[1H]{gillman2013rings}
    compact Hausdorff
        space, i.e.~where every cozero set has open closure,
        then~$[0,1]_{C(X)}$ is an~$\omega$-complete effect
            monoid~\cite[3N.5]{gillman2013rings}.
\end{example}

\section{Overview}
The main results of the paper are the following theorems:

\begin{theorem*}
    Let $M$ be an $\omega$-complete effect monoid.
    Then $M$ embeds into $M_1\oplus M_2$, where
    $M_1$ is an $\omega$-complete Boolean algebra,
    and $M_2 =[0,1]_{C(X)}$, 
    where $X$ is a basically disconnected compact Hausdorff space 
    (see Theorem~\ref{thm:omega-complete-classification}).
\end{theorem*}

\begin{theorem*}
    Let $M$ be a directed-complete effect monoid.
    Then $M \cong M_1 \oplus M_2$ where $M_1$ is a
    complete Boolean algebra and $M_2=[0,1]_{C(X)}$ for some
    extremally-disconnected compact Hausdorff space~$X$
    (see Theorem~\ref{mainthmdirectedcomplete}).
\end{theorem*}

By Example~\ref{ex:stone-embedding}, the Boolean algebra $M_1$ also embeds into a $[0,1]_{C(X_{M_1})}$, and hence we could `coarse-grain' the direct sums above and say that any $\omega$-complete effect monoid embeds into the unit interval of $C(X_{M_1} + X)$, where $+$ is the disjoint union of the topological spaces. 
This observation suggests a Stone-type duality that we discuss in more detail in the conclusion.

Other results for an $\omega$-complete effect monoid $M$ that either follow directly from the above theorems, or are proven 
along the way are the following:
\begin{itemize}
	\item $M$ is a lattice.
	\item $M$ is an \emph{effect divisoid}~\cite{basthesis}.
	\item The multiplication in $M$ is \emph{normal}: $a\cdot \bigvee S = \bigvee a\cdot S$.
	\item If~$M$ is convex (as an effect algebra),
            then scalar multiplication is \emph{homogeneous}:
        $\lambda (a \cdot b) = (\lambda a) \cdot b = a \cdot(\lambda b)$
            for any~$\lambda \in [0,1]$ and~$a,b \in M$.
    \item $M$ is commutative.
    \item If $M$ has no non-trivial zero-divisors (i.e.~$a\cdot b = 0$, 
    implies $a=0$ or $b=0$), then $M$ is isomorphic to $[0,1]$, $\{0,1\}$ or $\{0\}$.
\end{itemize}

It should be noted that the scalars in a $\omega$-effectus satisfying \emph{normalisation} have no non-trivial zero-divisors~\cite{kentathesis} and hence using the last point above, we have completely characterised the scalars in such $\omega$-effectuses, splitting them up into trivial, Boolean and convex effectuses.

The paper is structured as follows. In Section~\ref{sec:basicresults} we recover and prove some basic results regarding effect algebras/monoids. 
Then in Section~\ref{sec:floorceiling} we will show that in any $\omega$-complete effect monoid~$M$,
we can define a kind of partial division operation which turns it
into a \emph{effect divisoid}. 
Using this division we show that the multiplication must be normal.
Then in Section~\ref{sec:boolhalvesconvexity} we study idempotents that are either \emph{Boolean}, meaning that all elements below
$p$ must also be idempotents, or \emph{halvable}, meaning that there is an $a\in M$
such that $p=a\ovee a$.  
We establish that an $\omega$-complete 
effect monoid where $1$ is Boolean must
be a Boolean algebra, while if $1$ is halvable then it must be \emph{convex}.
In Section~\ref{sec:embeddingtheorem} we show that a maximal collection
of orthogonal idempotents of~$M$ can be found that consists of a mix
of halvable and Boolean idempotents. 
The corner $pM$ associated to such an idempotent
will either be convex (if~$p$ is halvable)
or Boolean (if~$p$ is Boolean).
Using normality of multiplication we show that $M$ embeds
into the direct sum of the corners associated to these idempotents.
Letting~$M_1$ be the direct sum of the Boolean corners,
and~$M_2$ be the direct sum of the convex corners,
we see that~$M$ embeds into~$M_1\oplus M_2$,
where~$M_1$ is Boolean and~$M_2$ is convex.
In Section~\ref{sec:OUS}, we recall some results regarding \emph{order unit spaces}
and use Yosida's representation theorem to show that a convex 
$\omega$-complete effect monoid must be isomorphic to 
the unit interval of a $C(X)$.
Then in Section~\ref{sec:maintheorems} we collect all the results and prove our main theorems. 
Finally in Section~\ref{sec:conclusion} we conclude and discuss some future work and open questions.


\section{Basic results}\label{sec:basicresults}

We do not assume any commutativity of the product in an effect
monoid. Nevertheless, some commutativity comes for free.

\begin{lemma}\label{lem:ssperpcommute}
    For any~$a \in M$ in an effect monoid~$M$,
        we have~$a\cdot a^\perp = a^\perp \cdot a$.
\end{lemma}
\begin{proof}
    $a^2 \ovee  (a^\perp
    \cdot a) = (a\ovee a^\perp) \cdot a  = 1\cdot a = a = a\cdot 1 = a \cdot (a\ovee a^\perp) = a^2 \ovee  (a
    \cdot a^\perp)$. Cancelling $a^2$ on both sides gives the desired
    equality.
\end{proof}

\begin{lemma}
    \label{lem:idempotentiff}
    An element $p\in M$ is an idempotent if and only if $p\cdot p^\perp = 0$.
\end{lemma}
\begin{proof}
    $p = p\cdot 1 = p\cdot (p\ovee p^\perp) = p^2 \ovee p\cdot p^\perp$. Hence $p=p^2$ if and only if $p\cdot p^\perp = 0$.
\end{proof}

\begin{lemma}\label{lem:preserveunderidempotent}
    For~$a, p \in M$ with~$p^2=p$, we have
    \begin{equation*}
            p \cdot a \ =\  a \quad \iff \quad
            a \cdot p \ =\  a \quad \iff \quad
            a \ \leq\  p.
    \end{equation*}
\end{lemma}
\begin{proof}
Assume $a\leq p$.
Then $a\cdot p^\perp \leq p\cdot p^\perp = 0$, so
    that~$a \cdot p^\perp = 0$. Similarly $p^\perp \cdot a = 0$.
     Hence~$a=a\cdot 1=a\cdot (p\ovee p^\perp) = a\cdot p \ovee a \cdot p^\perp
        = a\cdot p$. Similarly~$p \cdot a = a$.

Now assume~$p \cdot a = a$.
Then immediately~$a = p\cdot a \leq p \cdot 1 = p$.
    The final implication (that~$a \cdot p = a \implies a \leq p$) is proven
    similarly.
\end{proof}

\begin{lemma}\label{lem:idempotentscommute}
    Let $M$ be an effect monoid with idempotent~$p \in M$.
Then~$p\cdot a = a \cdot p$ for any~$a \in M$.
\end{lemma}
\begin{proof}
Clearly~$p\cdot a \leq p \cdot 1 = p$
    and so by Lemma~\ref{lem:preserveunderidempotent}
    $p \cdot a \cdot p = a \cdot p$.
    Similarly~$a \cdot p \leq p$ and so~$p \cdot a \cdot p = p \cdot a$.
    Thus~$p\cdot a = p\cdot a\cdot p = a \cdot p$, as desired.
\end{proof}
\begin{corollary}\label{cor:corneriso}
    Let~$M$ be an effect monoid with idempotent~${p \in M}$.
    The map~$e \mapsto (p \cdot e, p^\perp \cdot e)$
    is an isomorphism~$M \cong pM \oplus p^\perp M$.
\end{corollary}

The following two lemmas are simple observations
that will be used several times.
\begin{lemma}
\label{lem:forcing}
Let~$a\leq b$ be elements of an effect algebra~$E$.
If~$ b\ovee b'\leq a\ovee a'$ for some~$a'\leq b'$ from~$E$,
    then~$a=b$ (and~$a'=b'$).
\end{lemma}
\begin{proof}
    Since~$a\leq a'$ and~$b\leq b'$,
    we have~$a\ovee a'\leq b\ovee b'$,
    and so~$a\ovee a' = b\ovee b'$.
Then  $0 = (b\ovee b') \ominus (a\ovee a')=
    (b\ominus a)\ovee (b'\ominus a')$,
    yielding $b\ominus a= 0$
    and $b'\ominus a'=0$,
so~$b=a$ and~$b'=a'$.
\end{proof}

\begin{lemma}\label{lem:summableunderidempotent}
Let~$p$ be an idempotent from an effect monoid~$M$,
    and let~$a,b\leq p$ be elements below~$p$.
If~$a\ovee b$ exists,
then~$a\ovee b\leq p$.
\end{lemma}
\begin{proof}
    Since $a\leq p$, we have $a \cdot p^\perp = 0$,
    and similarly, $b\cdot p^\perp=0$.
    But then $(a\ovee b)\cdot p^\perp
    = 0$, and hence $(a \ovee b)\cdot p = a\ovee b$. By
    Lemma~\ref{lem:preserveunderidempotent} we then have $a\ovee b\leq
    p$.
\end{proof}

We defined directed set to mean upwards directed. Using the fact
that $a \mapsto a^\perp$ is an order anti-isomorphism, a directed-complete
effect algebra also has all infima of downwards directed (or `filtered') 
sets (and similarly for countable infima in a $\omega$-complete effect algebra).

Recall that given an element~$a$ of an
ordered group~$G$
a subset~$S$ of~$G$ has a supremum $\bigvee S$ in~$G$
if and only if $\bigvee_{s\in S} a+s$ exists,
which follows immediately from 
the observation that~$a+(\ )\colon G\to G$ is an order isomorphism.
For effect algebras
the situation is a bit more complicated,
and we only have the implications
mentioned in the lemma below.
We will see in Corollary~\ref{cor:sumomegaem}
that the situation improves
somewhat
for $\omega$-complete effect monoids.
\begin{lemma}\label{lem:additionisnormal}
Let~$x$ be an element and $S$ a non-empty
subset of an effect algebra~$E$.
    If $S\subseteq [0,x^\perp]_E$, then
\begin{alignat*}{3}
    \text{$\bigvee S$ exists}
    \quad&\implies\quad&
    x\ovee \bigvee S &\,=\,
    \bigvee x\ovee S\text{, and}
    \\
    \text{$\bigwedge x\ovee S$ exists}
    \quad&\implies\quad&
x\ovee \bigwedge S &\,=\,
    \bigwedge x\ovee S \text{.}
    \\
    \intertext{Here ``$=$'' means also that the 
    sums, suprema and infima on either side exist.
    Similarly, if $S\subseteq [x,1]_E$, then}
    \text{$\bigwedge_{s\in S} s\ominus x$ exists}
    \quad&\implies\quad&
\bigl(\bigvee S\bigr) \ominus x
     &\,=\, 
    \bigvee_{s\in S} s \ominus x\text{, and}
    \\
    \text{$\bigwedge S$ exists}
    \quad&\implies\quad&
\bigl(\bigwedge S\bigr) \ominus x
     &\,=\, 
    \bigwedge_{s\in S} s \ominus x\text{.}
    \\
\intertext{Moreover, if $S\subseteq[0,x]_E$, then}
    \text{$\bigvee S$ exists}
    \quad&\implies\quad&
x \ominus \bigvee S
     &\,=\,
    \bigwedge x\ominus S\text{, and}
    \\
    \text{$\bigvee x \ominus S$ exists}
    \quad&\implies\quad&
x \ominus \bigwedge S
     &\,=\, 
    \bigvee x\ominus S\text{.}
\end{alignat*}
\end{lemma}
\begin{proof}
    Note that $a \mapsto x\ovee a$
    gives an order isomorphism~$[0,x^\perp]_E\to [x,1]_E$
    with inverse~$a \mapsto a \ominus x$.
    Whence~$x\ovee(\ )$ preserves and reflects
    all infima and suprema \emph{restricted to}~$[0,x^\perp]_E$ and~$[x,1]_E$.
Surely,
given elements $a\leq b$ from~$E$,
and a subset~$S$ of the interval~$[a,b]_E$,
it is clear
that any supremum (infimum)
    of~$S$ in~$E$ will be the supremum (infimum) of~$S$ in~$[a,b]_E$ too
    (using here that~$S$ is non-empty).
The converse does not always hold,
    but when $S$ has a supremum in $[a,1]_E$,
    then this is the supremum in~$E$ too
    (and when~$S$ has an infimum in $[0,b]_E$,
    then this is the infimum in~$E$ too).
These considerations yield the first four equations.
For the latter two
we just add the observation
    that $x\ominus (\ )$
    gives an order reversing isomorphism~$[0,x]_E\to[0,x]_E$.
%
\end{proof}

We can now prove a few basic yet useful facts of $\omega$-complete effect monoids. 
These lemmas deal with elements that are summable with
    themselves: elements $a$ such that $a\perp a$ which means that
    $a\ovee a$ is defined. For $n\in \N$ we will use the notation $na
    = a\ovee \ldots \ovee a$ for the $n$-fold sum of $a$ with itself
    (when it is defined).
We study these self-summable elements
    to  be able to define a ``$\frac{1}{2}$''
    in some effect monoids later on.

\begin{lemma}\label{lem:selfsummable}
    For any~$a \in M$ in some effect monoid~$M$,
    the element~$a \cdot a^\perp$ is summable with itself.
\end{lemma}
\begin{proof}
Since $1=1\cdot 1 = (a\ovee a^\perp)\cdot(a\ovee a^\perp)
    = a\cdot a\ovee a\cdot a^\perp \ovee a^\perp \cdot a \ovee 
    a^\perp\cdot a^\perp$,
    and $a\cdot a^\perp = a^\perp \cdot a$ by Lemma~\ref{lem:ssperpcommute},
    we see that~$a\cdot a^\perp\ovee a\cdot a^\perp$  indeed exists.
\end{proof}

\begin{lemma}\label{lem:archemedeanomegadirectedcomplete}\label{lem:asummablepowerszero}\label{lem:nonilpotents}
    Let $a$ be an element of an $\omega$-complete effect monoid $M$.
    \begin{enumerate}
        \item If $na$ exists for all $n$ then $a=0$.
        \item If $a^2 = 0$ then $a=0$.
        \item If $a\perp a$ then $\bigwedge_n a^n = 0$.
    \end{enumerate}
\end{lemma}
\begin{proof}
    For point 1, we have $a\ovee \bigvee_n na
    =\bigvee_n a\ovee na
    = \bigvee_n (n+1)a = \bigvee_n na$,
    and so~$a=0$.

    For point 2, since $a^2 = 0$ we have $a = a\cdot 1 = a\cdot (a\ovee a^\perp) = a\cdot a^\perp$, and hence (because of
    Lemma~\ref{lem:selfsummable}) $a$ is summable with itself. But
    furthermore~$(a\ovee a)^2 = 4a^2 = 0$, and so $(a\ovee a)^2=0$.
    Continuing in this fashion,
        we see that~$2^n a$ exists for every~$n \in \N$
        and~$(2^n a)^2 = 0$.
        Hence, for any~$m \in \N$
        the sum~$m a$ exists 
    so that by the previous point
    $a = 0$.

    For point 3, write~$b \equiv \bigwedge_n a^n$.
    As $(2a)^n = 2^n a^n$ and $b\leq a^n$ we see that
    $2^n b$ is defined. But this is true for all $n$, and so again by the point 1, $b=0$.
\end{proof}





\section{Floors, ceilings and division}\label{sec:floorceiling}

In this section we will see that any $\omega$-complete effect monoid has
\emph{floors} and \emph{ceilings}. These are respectively the largest idempotent
below an element and the smallest idempotent above an element.
We will also construct a ``division'': for $a\leq b$ we will find 
an element $a/b$ such that $(a/b)\cdot b = a$. 

Then using ceilings and this division we will show that multiplication in a
$\omega$-complete effect monoid is always normal, i.e.~that 
$b\cdot \bigvee S = \bigvee b\cdot S$
for non-empty~$S$ for which~$\bigvee S$ exists. This technical result will be frequently used in the remaining sections.

\begin{definition}
    Let~$(x_i)_{i \in I}$ be a (potentially infinite)
    family of elements from an effect algebra~$E$.
    We say that the sum~$\bigovee_{i \in I} x_i$ exists
        if for every finite subset~$S \subseteq I$
        the sum~$\bigovee_{i \in S} x_i$ exists
        and the supremum~$\bigvee_{\text{finite }S \subseteq I}
        \bigovee_{i \in S} x_i$ exists
            as well.
    In that case we write $\bigovee_{i \in I} x_i \equiv
        \bigvee_{\text{finite }S \subseteq I}
        \bigovee_{i \in S} x_i$.
\end{definition}
\begin{lemma}
    \label{lem:proto-ceilfloor}
Given~$a \in M$ for an effect monoid~$M$, we have
\begin{equation*}
    \textstyle
    (a^N)^\perp\ =\ a^\perp\ \ovee\  a^\perp \cdot a\ \ovee\  a^\perp \cdot a^2
    \ \ovee\  \dotsb\ \ovee\  a^\perp \cdot a^{N-1}
\end{equation*}
for every natural number~$N$.
\end{lemma}
\begin{proof}
From the computation
    \begin{equation*}
\begin{alignedat}{3}
1 \ &=\ a^\perp \,\ovee\,a
\\
    &=\  a^\perp \,\ovee\,(a^\perp\ovee a)\cdot a \\
    &\qquad\ \equiv\  a^\perp\,\ovee\,a^\perp \cdot a \,\ovee\,a^2
\\
    &=\  a^\perp \,\ovee\,a^\perp \cdot a \,\ovee\,(a^\perp\ovee a) \cdot a^2 \\
    & \qquad \ \equiv\  a^\perp\,\ovee\,a^\perp\cdot a
    \,\ovee\, a^\perp\cdot  a^2\,\ovee\,
    a^3
\\
    &\ \vdots
\\
    &=\ \Bigl(\,\bigovee_{n=0}^{N-1} a^\perp \cdot a^n\,\Bigr)\,\ovee\, a^N
\end{alignedat}
\end{equation*}
the result follows immediately.
\end{proof}
\begin{corollary}
    The sum~$\bigovee_{n=0}^\infty a^\perp \cdot a^n$ exists
        for any element~$a$ 
        from an~$\omega$-complete effect monoid~$M$.
\end{corollary}
\begin{definition}
Given an element~$a$ of an $\omega$-complete
effect monoid~$M$
    \begin{equation*}
        \ceil{a}\ \equiv\ \bigovee_{n=0}^\infty a\cdot (a^\perp)^n
        \qquad\text{and}\qquad
        \floor{a}\ \equiv \ \bigwedge_{n=0}^\infty a^n
    \end{equation*}
    are called the \Define{ceiling} of a and the \Define{floor}
    of~$a$, respectively.
\end{definition}
We list
some basic properties of~$\ceil{a}$
and~$\floor{a}$ in Proposition~\ref{prop:ceilfloor},
after we have made the observations
necessary to establish them.

\begin{lemma}
    \label{lem:infaperpan}
Given an element~$a$ of an $\omega$-complete
effect monoid~$M$,
we have $\bigwedge_n a^\perp\cdot  a^n = 0$.
\end{lemma}
\begin{proof}
Write~$b \equiv \bigwedge_n a^\perp \cdot a^n$.
Since~$a$ and~$a^\perp$ commute by Lemma~\ref{lem:ssperpcommute},
    we compute
\begin{equation*}
1\ =\ 1^n\ =\ (a^\perp \ovee a)^n
    \ =\ \bigovee_{k=0}^n \binom{n}{k}\bigl(\, (a^\perp)^k\cdot a^{n-k}
    \,\bigr),
\end{equation*}
and in particular see
    that the sum $\binom{n}{1} (a^\perp \cdot a^{n-1})\equiv 
    n(a^\perp \cdot a^{n-1})$ exists.
Because $b\leq a^\perp \cdot a^{n-1}$,
the $n$-fold sum~$nb$ exists too and hence~$b=0$
    by Lemma~\ref{lem:archemedeanomegadirectedcomplete}.
\end{proof}
\begin{lemma}
\label{lem:flooraa}
    We have~$\floor{a}=\floor{a}\cdot a = a \cdot \floor{a}$
for any element~$a$ of an $\omega$-complete
effect monoid~$M$.
\end{lemma}
\begin{proof}
Using Lemmas~\ref{lem:ssperpcommute} and~\ref{lem:infaperpan} we compute $\floor{a} \cdot a^\perp
    = (\bigwedge_n a^n) \cdot a^\perp
\leq \bigwedge_n  a^n \cdot a^\perp 
= \bigwedge_n a^\perp \cdot a^n 
= 0$,
    and so~$\floor{a}\cdot a=\floor{a}$.
The other identity
    has a similar proof.
\end{proof}


\begin{lemma}
\label{lem:zerodivsum}
Given elements $a, b_1, b_2,\dotsc$
of a $\omega$-complete effect monoid~$M$
such that~$\bigovee_n b_n$ exists,
and~$a\cdot b_n=0$ for all~$n\in \N$,
we have~$a \cdot \bigovee_n b_n=0$.
\end{lemma}
\begin{proof}
    Writing~$s_N\equiv\bigovee_{n=1}^N b_n$,
    we have~$s_1\leq s_2\leq \dotsb$
    and~$a\cdot s_n=0$ for all~$n$.
    Since~$s_n=(a\ovee a^\perp)\cdot s_n
    = a \cdot s_n \ovee a^\perp \cdot s_n = a^\perp \cdot s_n$
    for all~$n \in \N$,
    we have
\begin{equation*}
    \textstyle \bigvee_n s_n \ = \ \bigvee_n a^\perp \cdot s_n
    \ \leq\  a^\perp \cdot \bigvee_n s_n
    \ \leq\  \bigvee_n s_n,
\end{equation*}
which implies that~$a^\perp\cdot \bigvee_ns_n= \bigvee_n s_n$,
and thus~$a\cdot \bigovee_n b_n \equiv a \cdot \bigvee_n s_n=0$.
\end{proof}
\begin{proposition}
\label{prop:ceilprod}
Given elements $a$ and~$b$ of an $\omega$-complete effect monoid~$M$,
$$
    a\cdot b\,=\,0\quad\implies\quad a\cdot\ceil{b}\,=\,0.
$$
\end{proposition}
\begin{proof}
If~$a\cdot b=0$,
then also $a\cdot b\cdot (b^\perp)^n=0$
    for all~$n$.
Hence by Lemma~\ref{lem:zerodivsum}
$a\cdot \ceil{b}\equiv a\cdot \bigovee_{n=1}^\infty b\cdot (b^\perp)^n
    =0$.
\end{proof}
\begin{proposition}
\label{prop:ceilfloor}
Let~$a$ be an element of an $\omega$-complete
effect monoid~$M$.
\begin{enumerate}
\item
    \label{prop:ceilfloor-floor}
The floor~$\floor{a}$ of~$a$
        is an idempotent with~$\floor{a}\leq a$.
        In fact, $\floor{a}$
        is the greatest idempotent below~$a$.
\item
    \label{prop:ceilfloor-ceil}
The ceiling~$\ceil{a}$ of~$a$
is the least idempotent above~$a$.
\item
    \label{prop:ceilfloor-duality}
    We have
        $\ceil{a}^\perp = \floor{a^\perp}$
        and
        $\floor{a}^\perp = \ceil{a^\perp}$.
\end{enumerate}
\end{proposition}
\begin{proof}
Point~\ref{prop:ceilfloor-duality}
    follows from Lemma~\ref{lem:proto-ceilfloor}.
Concerning point \ref{prop:ceilfloor-floor}:
Since $\floor{a} \cdot a^\perp = 0$
    (by Lemma~\ref{lem:flooraa})
    we have $\floor{a}\cdot \ceil{a^\perp}=0$
    by Proposition~\ref{prop:ceilprod},
    and so $\floor{a}\cdot \floor{a}^\perp=0$
    because
     $\floor{a}^\perp = \ceil{a^\perp}$
     by point~\ref{prop:ceilfloor-duality}. Hence $\floor{a}$ is an idempotent.
Also, 
since~$\floor{a}=\bigwedge_n a^n$,
we clearly have~$\floor{a}\leq a$.
Now,
if~$s$ is an idempotent in~$M$
with~$s\leq a$,
then~$s=s^n\leq a^n$,
and so~$s\leq \bigwedge_n a^n\equiv \floor{a}$.
Whence~$\floor{a}$ is the greatest idempotent below~$a$.
Point~\ref{prop:ceilfloor-ceil}
now follows easily from~\ref{prop:ceilfloor-floor},
    since~$\ceil{\,\cdot\,}$ is the dual of~$\floor{\,\cdot\,}$
    under the order anti-isomorphism $(\,\cdot\,)^\perp$.
\end{proof}
\begin{lemma}\label{lem:sumofceil}
$\ceil{a\ovee b}=\ceil{a}\vee \ceil{b}$
for all summable
    elements $a$ and~$b$ of an $\omega$-complete effect monoid~$M$
    (that is,
    $\ceil{a\ovee b}$ is the supremum of~$\ceil{a}$ and~$\ceil{b}$).
\end{lemma}
\begin{proof}
Since~$\ceil{a\ovee b}\geq a\ovee b \geq a$,
    we have $\ceil{a\ovee b } \geq \ceil{a}$,
    and similarly, $\ceil{a\ovee b}\geq \ceil{b}$.
Let~$u$ be an upper bound of~$\ceil{a}$ and~$\ceil{b}$;
    we claim that~$\ceil{a\ovee b}\leq u$.
    Since~$\ceil{a}\leq u$ and~$\ceil{b}\leq u$, 
    we have $a\leq \ceil{a}\leq \floor{u}$ and
    $b\leq \ceil{b}\leq \floor{u}$,
    and so~$a\ovee b\leq \floor{u}$ by Lemma~\ref{lem:summableunderidempotent}.
    Whence~$\ceil{a\ovee b}\leq \floor{u}\leq u$.
\end{proof}
Any $\omega$-complete effect monoid is a \emph{lattice effect algebra}~\cite{rievcanova2000generalization}:
\begin{theorem}\label{thm:latticeemon}
Any pair of elements $a$ and $b$ from an $\omega$-complete
effect monoid~$M$
has an infimum, $a\wedge b$, given by
$$
    a\wedge b\,=\, \bigovee_{n=1}^\infty a_n\cdot b_n
    \ \text{where}\ 
    \left[\ 
    \begin{alignedat}{3}
        a_1\,&=\, a \quad&& a_{n+1} \,&=\, a_n\cdot b_n^\perp \\
        b_1\,&=\, b && b_{n+1} \,&=\, a_n^\perp \cdot b_n
    \end{alignedat}\right..
$$
Consequently, any pair also has a supremum given by $a\vee b= (a^\perp \wedge b^\perp)^\perp$.
\end{theorem}
\begin{proof}
First order of business is showing 
    that the sum~$\bigovee_{n=1}^N a_n \cdot b_n$ exists for every~$N$.
    In fact,
    we'll show that $a\ominus \bigovee_{n=1}^N a_n\cdot b_n = a_{N+1}$
    for all~$N$,
    by induction.
Indeed, for~$N=1$,
we have~$a\ominus a\cdot b=a\cdot b^\perp = a_2$,
and if
$a\ominus \bigovee_{n=1}^N a_n\cdot b_n = a_{N+1}$
for some~$N$,
then~$a_{N+2} = a_{N+1} \cdot b_{N+1}^\perp
    = a_{N+1} \ominus a_{N+1}\cdot  b_{N+1}
    = (a\ominus \bigovee_{n=1}^N a_n\cdot b_n)\ominus a_{N+1}\cdot b_{N+1}
    = a\ominus \bigovee_{n=1}^{N+1} a_n \cdot b_n$.
    In particular,
    $\bigovee_{n=1}^\infty a_n \cdot b_n$ exists,
    and, moreover,
    $$
    a\ =\ \bigwedge_{m=1}^\infty a_m \ \ovee\ 
    \bigovee_{n=1}^\infty a_n \cdot b_n.
    $$
By a similar reasoning, we get
    $$
    b\ =\ \bigwedge_{m=1}^\infty b_m \ \ovee\ 
    \bigovee_{n=1}^\infty a_n \cdot b_n.
    $$
Already writing~$a\wedge b \equiv
    \bigovee_{n=1}^\infty a_n \cdot b_n$,
    we know at this point that
    $a\wedge b \leq a$ and~$a\wedge b\leq b$.
    It remains to be shown that~$a\wedge b$
    defined above is the greatest lower bound of~$a$ and~$b$.
    So let~$\ell\in M$ with $\ell\leq a$ and~$\ell \leq b$ 
    be given; we must show that~$\ell\leq a\wedge b$.


    As an intermezzo, we observe that~$\bigl(\bigwedge_n a_n\bigr)\cdot
    \bigl(\bigwedge_m b_m\bigr)=0$.
    Indeed, we have
    $\bigl(\bigwedge_n a_n\bigr)\cdot
    \bigl(\bigwedge_m b_m\bigr)
    \leq \bigwedge_n a_n\cdot b_n$,
    and
    $\bigwedge_n a_n \cdot b_n = 0$,
    because~$\bigovee_{n=1}^\infty a_n\cdot b_n$ exists (see Lemma~\ref{lem:archemedeanomegadirectedcomplete}).
    By Proposition~\ref{prop:ceilprod}
    it follows that $\bigl(\bigwedge_n a_n\bigr)\cdot
    \ceil{\bigwedge_m b_m} = 0$.
    Whence writing~$p \equiv \ceil{\bigwedge_m b_m}$,
    we have $p\cdot \bigwedge_n a_n \equiv \bigl(\bigwedge_n a_n\bigr)\cdot p
    = 0$ 
    using Lemma~\ref{lem:idempotentscommute}.
    Observing that $\bigwedge_n b_n\leq p$ and using Lemma~\ref{lem:preserveunderidempotent} we also have $p^\perp \cdot \bigwedge_n b_n = 0$.
    We then calculate  $p\cdot a
    =p\cdot \bigl(\,\bigwedge_n a_n \,\ovee\, a\wedge b\,\bigr) 
    = p\cdot (a\wedge b)$ and similarly
    $p^\perp \cdot b= p^\perp \cdot (a\wedge b)$.

    Returning to the problem of whether~$\ell\leq a\wedge b$,
    we have
   \begin{multline*}
\ell \ =\  
p\cdot \ell\,\ovee\, p^\perp\cdot  \ell
\ \leq\  p\cdot a\,\ovee\, p^\perp\cdot  b \\
    \ =\ p\cdot (a\wedge b) \,\ovee\, p^\perp\cdot (a\wedge b)
     \ = \ a\wedge b.
   \end{multline*}
    Whence~$a\wedge b$ is the infimum of~$a$ and~$b$.
\end{proof}
The presence of finite infima and suprema in $\omega$-complete
effect monoids prevents certain subtleties
around the existence of arbitrary suprema and infima.
\begin{corollary}
\label{cor:intervalsuprema}
Let~$a\leq b$ be elements of an $\omega$-complete
effect monoid~$M$,
    and let $S$ be a non-empty subset of~$[a,b]_M$.

    Then~$S$ has a supremum (infimum) in~$M$
    if and only if~$S$ has a supremum (infimum) in~$[a,b]_M$,
    and these suprema (infima) coincide.
\end{corollary}
\begin{proof}
It is clear that if~$S$ has a supremum in~$M$, then
    this is also the supremum in~$[a,b]_M$.
For the converse,
suppose that~$S$
has a supremum~$\bigvee S$ in $[a,b]_M$,
and let~$u$ be an upper bound for~$S$ in~$M$;
in order to show that~$\bigvee S$ is the supremum of~$S$
in~$M$ too,
we must prove that~$\bigvee S\leq u$.
    Note that~$b\wedge u$
    is an upper bound for~$S$.
    Indeed, given~$s\in S\subseteq [a,b]_M$
    we have~$s\leq b$, and~$s\leq u$,  so $s\leq b\wedge u$.
    Moreover, one easily sees
    that $b\wedge u \in [a,b]_M$
    using the fact that~$S$ is non-empty.
    Whence~$b\wedge u$ is an upper bound
    for~$S$ in $[a,b]_M$, and so
    $\bigvee S\leq b\wedge u\leq u$,
    making~$\bigvee S$ the supremum of~$S$ in~$M$.
    Similar reasoning applies
    to infima of~$S$.
\end{proof}
\begin{corollary}
\label{cor:sumomegaem}
Given an element~$a$
and a non-empty subset~$S$ 
of an $\omega$-complete
effect monoid~$M$
such that $a\ovee s$ exists for all~$s\in S$,
\begin{enumerate}
\item
the supremum~$\bigvee S$ exists
iff $\bigvee a\ovee S$ exists,
and in that case $a\ovee \bigvee S=\bigvee a\ovee S$;
\item
the infimum~$\bigwedge S$ exists
iff $\bigwedge a\ovee S$ exists,
and in that case $a\ovee \bigwedge S=\bigwedge a\ovee S$.
\end{enumerate}
\end{corollary}
\begin{proof}
The map~$a\ovee(\ )\colon [0,a^\perp]_M\to [a,1]_M$,
being an order isomorphism,
preserves and reflects suprema and infima.
    Now apply Corollary~\ref{cor:intervalsuprema}.
\end{proof}

Now that we know more about the existence of suprema and infima, we set our sights on proving that multiplication interacts with suprema and infima as desired, namely that it preserves them. To do this we introduce a partial division operation.
\begin{definition}
Given elements~$a\leq b$ of an $\omega$-complete
effect monoid, set
\begin{equation*}
    a/b \ \equiv \ \bigovee_{n=0}^\infty a \cdot (b^\perp)^n.
\end{equation*}
    Note that the sum exists,
    because $\bigovee_{n=0}^N
    a\cdot (b^\perp)^n 
    \leq {\bigovee_{n=0}^\infty b \cdot (b^\perp)^n}
    \equiv\ceil{b}$ for all~$N$.
\end{definition}
\begin{lemma}
\label{lem:div}
Let $b$ be an element of an $\omega$-complete
effect monoid~$M$.
\begin{enumerate}
\item
\label{lem:div-adiva}
$b/b = \ceil{b}$.
\item
\label{lem:div-additive}
$(a_1\ovee a_2)/ b = a_1/b\,\ovee a_2/b$
for all summable~$a_1,a_2\in M$ with $a_1\ovee a_2\leq b$.
\item
\label{lem:div-abdivb}
$(a\cdot b)/b = a\cdot \ceil{b}$
for all~$a\in M$.
\item
\label{lem:div-adivbb}
$(a/b)\cdot b = a$
for all~$a\in M$ with~$a\leq b$.
\item
    \label{lem:div-mb}
        $\{ a \cdot b; a \in M \} \equiv  Mb  = [0,b]_M \equiv \{a; a\in M; a \leq b\}$.
\item
    \label{lem:div-iso}
        The maps $a\mapsto a\cdot b,b\cdot a\colon M\ceil{b} \to M b$
are order isomorphisms.
\end{enumerate}
\end{lemma}
\begin{proof}
    Points~\ref{lem:div-adiva} and~\ref{lem:div-additive} are easy,
and left to the reader.
Concerning~\ref{lem:div-abdivb},
first note that
\begin{equation*}
    (a\cdot b)/b\ \equiv\ 
    \bigovee_{n=0}^\infty a \cdot b \cdot (b^\perp)^n
\ \leq\ 
    a \cdot \bigovee_{n=0}^\infty b \cdot (b^\perp)^n
    \ = \ a\cdot \ceil{b}.
\end{equation*}
Thus $(a \cdot b)/b\leq a\cdot \ceil{b}$.
    Since similarly $(a^\perp \cdot b)/b\leq a^\perp\cdot \ceil{b}$,
    we get, using~\ref{lem:div-additive},
\begin{equation*}
\ceil{b}\ = \ 
b/b\ = \ 
    (a\cdot b)/b\,\ovee\,(a^\perp \cdot b)/b
    \ \leq\ a\cdot \ceil{b}\,\ovee\,a^\perp \cdot\ceil{b}
    \ =\ \ceil{b},
\end{equation*}
forcing~$(a\cdot b)/b = a\cdot \ceil{b}$
    (see Lemma~\ref{lem:forcing}).
For point~\ref{lem:div-adivbb},
note that given~$a,b \in M$ with~$a \leq b$
we have
    $a=a\cdot \ceil{b}$ (by Lemma~\ref{lem:preserveunderidempotent}, 
    since $a\leq b\leq \ceil{b}$,) and so
    \begin{alignat*}{3}
        a 
         \ =\ a\cdot \ceil{b} 
        & \ =\ (a\cdot b)/b &\qquad& \text{by point~\ref{lem:div-abdivb}}\\
        & \ \equiv\  \bigovee_{n=0}^\infty  a\cdot b\cdot (b^\perp)^n \\
        & \ =\  \bigovee_{n=0}^\infty a\cdot (b^\perp)^n \cdot b 
                    && \text{by Lemma~\ref{lem:ssperpcommute}} \\
        & \ \leq\  \Bigl(\bigovee_{n=0}^\infty a\cdot (b^\perp)^n\Bigr) \cdot b  \\
        & \ =\ (a/b)\cdot b.
    \end{alignat*}
    Since similarly $b\ominus a\leq ((b\ominus a)/b)\cdot b$,
    we get
\begin{multline*}
    b\ =\ a\ovee (b\ominus a) 
    \ \leq\ 
    (a/b)\cdot b\,\ovee\,((b\ominus a)/b)\cdot b \\
     \ =\ (b/b)\cdot b =\ceil{b}\cdot b\ =\ b,
\end{multline*}
which forces~$a=(a/b)\cdot b$.
For point~\ref{lem:div-mb},
    note that~$Mb, bM\subseteq [0,b]_M$
    since~$b\cdot a, a\cdot b\leq b$ for all~$a\in M$,
    and~$[0,b]_M\subseteq Mb$,
    because~$a=(a/b)\cdot b$ for any~$a\in [0,b]_M$
    by point~\ref{lem:div-adivbb}.

Finally, 
concerning
    point~\ref{lem:div-iso}:
    the maps 
     $a\mapsto a\cdot b\colon M \ceil{b}
    \to M b$
and  $a\mapsto a/b\colon M b
    \to M \ceil{b}$
    are clearly order preserving,
    and each other's inverse
    by points~\ref{lem:div-abdivb} and~\ref{lem:div-adivbb},
    and thus order isomorphisms.
    The proof that~$a\mapsto b\cdot a\colon M \ceil{b} \to M b$
    is an order isomorphism follows along entirely similar lines,
    but involves $b\backslash a$ defined by 
    $b\backslash a\equiv \bigovee_n (b^\perp)^n \cdot a$
    and uses the fact that~$bM = [0,b]_M = M b$.
\end{proof}
\begin{remark}
From the previous lemma it follows
    that any~$\omega$-complete effect monoid 
    is a so called \emph{effect divisoid} \cite[§195]{basthesis}.
The converse is false:
    later on we will show that any~$\omega$-complete
    effect monoid is commutative,
        but there exists a non-commutative effect divisoid.
    \cite[Ex.~4.3.9]{kentathesis}\footnote{%
        Cho shows that there is a non-commutative \emph{division effect monoid}.
        Any division effect monoid is an effect divisoid as well.}.
\end{remark}

Finally, we can prove that multiplication is indeed normal:
\begin{theorem}\label{thm:multisnormal}
    \label{prop:bsupinf}
    Let $b$ and~$b'$ be elements
    of an $\omega$-complete
    effect monoid~$M$,
    and let $S\subseteq M$ be any 
    (potentially uncountable or non-directed) non-empty subset.
    \begin{enumerate}
        \item \label{prop:bsupinf-supp}
            If $\bigvee S$ exists, then
            so does $\bigvee_{s\in S} b\cdot s \cdot b'$,
            and $b\cdot (\bigvee S ) \cdot b'
            = \bigvee_{s\in S} b\cdot s\cdot b'$.
        \item
            \label{prop:bsupinf-inf}
        If $\bigwedge S$ exists, then
            so does $\bigwedge_{s\in S} b\cdot s \cdot b'$,
            and $b\cdot (\bigwedge S) \cdot b' = \bigwedge_{s\in S} 
            b\cdot s\cdot b'$.
    \end{enumerate}
\end{theorem}
\begin{proof}
Suppose that~$\bigvee S$ exists.
    We will prove that~$b\cdot \bigvee S=\bigvee_{s\in S} b\cdot s$,
and leave the remainder to the reader.
    Note that  $b\cdot(\ )\colon [0,\ceil{b}]_M\to [0,b]_M$,
    being an order isomorphism by Lemma~\ref{lem:div}\eqref{lem:div-iso},
    preserves suprema and infima.
    The set~$S$ need, however, not be part of $[0,\ceil{b}]_M$,
    so we consider instead of~$b$
    the element $b'\equiv  b\ovee \ceil{b}^\perp$,
    for which~$\ceil{b'}=\ceil{b}\vee \ceil{b}^\perp =1$
    by Lemma~\ref{lem:sumofceil}.
    We then get an order isomorphism~${b'\cdot(\ )\colon M\to [0,b']_M}$,
    which preserves suprema,
    so that~$b'\cdot \bigvee S$
is the supremum of $b'\cdot S$ in $[0,b']_M$,
    and hence in~$M$, by Corollary~\ref{cor:intervalsuprema}.
    Then
\begin{alignat*}{3}
    (b\ovee \ceil{b}^\perp)\cdot \bigvee S 
    \ &=\ \bigvee_{s\in S} (b\ovee \ceil{b}^\perp)\cdot s \\
    \ &\leq\ \bigvee_{s\in S} b\cdot  s 
    \ \ovee\ \bigvee_{s'\in S} \ceil{b}^\perp \cdot   s' \\
    \ &\leq\ b\cdot \bigvee S \ \ovee\ \ceil{b}^\perp\cdot \bigvee S \\
    \ & = \ (b\ovee \ceil{b}^\perp)\cdot\bigvee S
\end{alignat*}
forces $\bigvee_{s\in S} b\cdot s= b\cdot \bigvee S$
    (see Lemma~\ref{lem:forcing}).
\end{proof}

\section{Boolean algebras, halves and convexity}\label{sec:boolhalvesconvexity}
We are ready to study the two important types of idempotents
    in an effect monoid: those that are \emph{Boolean}
    and those that are \emph{halvable}.

\begin{definition}
We say that an element~$a$ of an effect monoid~$M$
    is \Define{Boolean} when each~$b\leq a$
    is idempotent. We say an effect monoid is Boolean
    when $1$ is Boolean.
\end{definition}

\begin{proposition}\label{prop:booleanalgebra}
    The set of idempotents~$P(M)$ of an effect monoid~$M$
        is a Boolean algebra.
        Thus an effect monoid is Boolean iff it is a Boolean algebra.
\end{proposition}
\begin{proof}
    First we will show that in fact~$p\cdot q = p\wedge q$
        for~$p,q \in P(M)$, where the infimum~$\wedge$ is taken in~$M$.
    Using Lemma~\ref{lem:idempotentscommute},
    we see~$(p \cdot q)^2 = p\cdot q\cdot p\cdot q = p\cdot
    p\cdot q\cdot q = p\cdot q$ and so~$p \cdot q$ is an idempotent.
    Let~$r \leq p, q$.
    Then~$r \cdot p = r$ and~$r \cdot q = r$
        so that~$r \cdot p \cdot q = r$, and
        hence~$r \leq p \cdot q$ by~Lemma~\ref{lem:preserveunderidempotent},
        which shows~$p \cdot q = p \wedge q$.
     As the complement is an order anti-isomorphism,
        we find~$p\vee q = (p^\perp \wedge q^\perp)^\perp$ and
    hence $P(M)$ is a complemented lattice.
    It remains to show that
    it satisfies distributivity: $p\wedge(q\vee r) = (p\wedge q)\vee
    (p\wedge r)$.
    By uniqueness of complements,
     it is easily shown that~$p\vee q
     = p \ovee (p^\perp \cdot q) = q \ovee ( p\cdot q^\perp)$.
    The remainder is a straightforward exercise in writing out
    the expressions $p\wedge(q\vee r)$ and $(p\wedge q)\vee
        (p\wedge r)$ and noting that they are equal.
\end{proof}

\begin{proposition}\label{prop:completelattice}
If~$M$ is $\omega$-complete,
    then the Boolean algebra of projections~$P(M)$ is $\omega$-complete.
\end{proposition}
\begin{proof}
    Let~$A \subseteq P(M)$ be a countable subset. Pick an enumeration of 
    its elements $p_1,p_2,\ldots$. Let $q_n$ be iteratively defined
    as $q_1 \equiv p_1$ and $q_n \equiv q_{n-1}\vee p_n$.
    Then the $q_n$ form an increasing sequence and hence it has
    a supremum $q$. We claim that $q$ is also the supremum of $A$.
    Of course $q\geq q_n \geq p_n$ and hence $q$ is an upper bound.
    Suppose that $r\geq p_n$ for all $n$. Then $r\geq q_1$, and hence by induction
    if $r\geq q_n$ then $r\geq q_n \vee p_n = q_{n+1}$. Hence also $r\geq q$. \qedhere
    
\end{proof}

\begin{proposition}\label{prop:sharpeffectmonoidisbooleanalgebra}
    Let $M$ be an $\omega$-complete Boolean effect monoid. 
    Then $M$ is an $\omega$-complete Boolean algebra.
\end{proposition}
\begin{proof}
    By Propositions~\ref{prop:booleanalgebra} and \ref{prop:completelattice} 
    $P(M)$ is an $\omega$-complete Boolean algebra. 
    But by assumption every element of $M$ is an idempotent, and hence $M=P(M)$.
\end{proof}

The counterpart to the Boolean effect monoids, are the \emph{halvable effect monoids}
\begin{definition}
We say that an element~$a$ of an effect algebra~$E$
    is \Define{halvable}
    when~$a=b\ovee b$ for some~$b\in E$. We say an effect algebra
    is halvable when $1$ is halvable.
\end{definition}

A halvable effect monoid actually has much more structure then might be apparent:
\begin{definition}\label{def:convexeffectalgebra}
    Let $E$ be an effect algebra. We say $E$ is \Define{convex} if
    there exists an action $\cdot: [0,1]\times E\rightarrow E$,
    where $[0,1]$ is the standard real unit interval, satisfying
    the following axioms for all $a,b\in E$ and~$\lambda, \mu \in
    [0,1]$:
    \begin{itemize}
        \item $\lambda\cdot (\mu\cdot a) = (\lambda\mu)\cdot a$.
    \item If $\lambda+\mu \leq 1$, then $\lambda \cdot a \perp
    \mu \cdot a$ and $\lambda \cdot a \ovee \mu \cdot a =
    (\lambda+\mu)\cdot a$.
        \item If $a\perp b$, then $\lambda\cdot a \perp \lambda \cdot b$ and $\lambda\cdot(a\ovee b) = \lambda\cdot a \ovee \lambda \cdot b$.
        \item $1\cdot a = a$.
    \end{itemize}
\end{definition}

In a convex effect monoid we will usually write
the convex action without any symbol in order to distinguish it 
from the multiplication coming from the monoid structure. 
So if $\lambda \in [0,1]$ is a real number and $a,b \in M$ is 
a convex effect monoid, then we write $\lambda(a\cdot b)$. 
Note that a priori it is not clear whether 
$\lambda (a\cdot b) = (\lambda a)\cdot b
    =   a \cdot (\lambda b)$.

\begin{proposition}\label{prop:dircompleteisconvex}
    Let $M$ be a halvable $\omega$-complete effect monoid.
        Then $M$ is convex.
\end{proposition}
\begin{proof}
    Pick any~$a \in M$ with~$a \ovee a = 1$.
    Let $q = \frac{m}{2^n}$ be a dyadic rational number with $0\leq m\leq
    2^n$. We define a corresponding element~$\cl{q} \in M$  by
        $\cl{q} = m a^n$, which is easily seen to be independent
        of the choice of~$m$ and~$n$.
    This yield an action~$(q, s) \mapsto \cl{q} \cdot s$ that satisfies
        all axioms of Definition~\ref{def:convexeffectalgebra}
        restricted to dyadic rationals.
    
    Assume~$\lambda \in (0,1]$.
    Pick a strictly increasing
        sequence~${0\leq q_1< q_2<\ldots}$ of dyadic rationals
        with~$\sup q_i = \lambda$.

        We will define~$\cl\lambda \in M$ by~$\bigvee_i \cl {q_i}$,
            but first we have to show that it is independent
                of the choice of the sequence
                    and that it coincides with the definition
                    just given for dyadic rationals.
        So assume~$0 \leq p_1 < p_2 < \ldots$
                is any other sequence of dyadic rationals
                with~$\sup p_i=\lambda$.
    For any~$p_i$ we can find a~$q_j$ with~$p_i \leq q_j$,
        so~$\cl {p_i} \leq \cl {q_j}$, hence~$\bigvee_i \cl{p_i}
                \leq \bigvee_j \cl{q_j}$.
        As the situation is symmetric between the sequences,
            we also have~$\bigvee_j \cl{q_j} \leq \bigvee_i \cl{p_i}$
            and so~$\bigvee_j \cl{q_j} = \bigvee_i \cl{p_i}$.
        Hence~$\cl \lambda$ is independent of the choice of sequence.
    Next, assume~$\lambda \equiv q$ is a non-zero dyadic rational.
    Pick~$m$ with~$2^{-m} \leq q$.
    Then~$q_n \equiv q - {2^{-(m+n)}}$
        is a sequence of dyadic rationals with~$\sup q_n = q$.
    We have~$\bigvee_n \cl {q - 2^{-(m+n)}}
                    = \cl q \ominus \bigwedge_n \cl{2^{-(m+n)}}
                    = \cl q \ominus \bigwedge_n a^{m+n}
                    = \cl q $,
                    (where in the last step we used Lemma~\ref{lem:asummablepowerszero}) so both definitions of~$\cl q$ coincide.
        As a result we are indeed justified
            to define~$\overline\lambda = \bigvee_i \cl {q_i}$.
 We can then define an action by $(\lambda, s)\mapsto \cl{\lambda}\cdot s$.
 As both addition and multiplication preserve suprema 
 by Theorem~\ref{thm:multisnormal}, it is straightforward to show 
 that this action indeed satisfies all the axioms of a convex action.
\end{proof}

\section{Embedding theorems}\label{sec:embeddingtheorem}

We now have what we need to show that any $\omega$-complete effect monoid embeds into a direct sum of a Boolean effect monoid and a halvable one, which lies at the heart of our results.

\begin{lemma}
\label{lem:ceilhalveable}
    Let $M$ be an $\omega$-complete effect monoid, and let $a$ be a halvable element.
    Then the ceiling~$\ceil{a}$
    is halvable as well.
\end{lemma}
\begin{proof}
Write~$a\equiv b\ovee b$.
We compute
\begin{multline*}
\textstyle
    \ceil{a} \ \equiv\ 
    \bigovee_n a  \cdot  (a^\perp)^n \ =\  
    \bigovee_n (b\ovee b) \cdot (a^\perp)^n \\
    \ =\  
\textstyle
    \bigl(\bigovee_n b \cdot (a^\perp)^n\bigr) \,\ovee\, \bigl(
    \bigovee_n b \cdot (a^\perp)^n \bigr),
\end{multline*}
    and hence it is indeed halvable.
\end{proof}

\begin{proposition}\label{prop:maxcollectionbooleanhalveable}
Each $\omega$-complete effect monoid~$M$
has a subset~$E \subseteq M$ 
\begin{enumerate}
\item
that is a maximal collection of non-zero orthogonal idempotents,
and
\item
such that each element of~$E$ is either halvable or Boolean.
\end{enumerate}
\end{proposition}
\begin{proof}
Let~$H$ be a maximal collection
of non-zero orthogonal halvable idempotents of~$M$,
and let~$E$ be a maximal collection
of non-zero orthogonal idempotents of~$M$
that extends~$H$.
(Such sets~$E$ and~$H$ exist by Zorn's Lemma).
By definition, $E$ is a maximal collection
of non-zero orthogonal idempotents of~$M$,
so the only thing to prove
is that each~$e\in E\backslash H$ is Boolean.
Hence, let~$a$ be an element of~$M$ below
some~$e\in E\backslash H$;
we must show that~$a$ is an idempotent.
    Note that~$2(a^\perp \cdot a)\equiv a^\perp \cdot a \ovee a^\perp\cdot a$ 
    (which exists by e.g.~Lemma~\ref{lem:selfsummable}) is halvable,
    and $2(a^\perp\cdot a) \leq e$,
    because $2(a^\perp \cdot a)\cdot e = 2(a^\perp \cdot a \cdot e)
     = 2(a^\perp\cdot a)$.
Then the idempotent $\ceil{2a\cdot a^\perp} \leq e$,
    which is halvable by Lemma~\ref{lem:ceilhalveable},
is orthogonal to all~$h\in H$
    (since it is below~$e$)
and must therefore be zero
lest it contradict the maximality of~$H$.
In particular, 
    $a^\perp\cdot a=0$ since $a^\perp \cdot a \leq \ceil{2a^\perp\cdot a}=0$,
    and so~$a$ is an idempotent by Lemma~\ref{lem:idempotentiff}.
Whence~$e$ is Boolean.
\end{proof}

Note that the only idempotent that is both Boolean and halvable is zero, and hence
each element in the above set is either Boolean \emph{or} halvable.

    

\begin{proposition}
Given a maximal orthogonal collection of non-zero idempotents~$E$
of an $\omega$-complete effect monoid~$M$,
the map
$$
    a\mapsto (a\cdot e)_e\colon M\longrightarrow \bigoplus_{e\in E} Me
$$
is an embedding of effect monoids.
\end{proposition}
\begin{proof}
    The map obviously maps $1$ to $1$, and preserves addition. 
    Hence it also preserves the complement and the order.
    By Lemma~\ref{lem:idempotentscommute} we have 
    $(a\cdot e)\cdot (b\cdot e) = (a\cdot b)\cdot (e\cdot e) =(a\cdot b)\cdot e$,
    and so the map also preserves the multiplication.
    It remains to show that the map is order reflecting. 
    Note that if we had~$\bigovee E = 1$,
    then for any $a$ by Theorem~\ref{thm:multisnormal} $a = a\cdot 1 = a\cdot \bigovee_{e\in E} e 
    = \bigovee_{e\in E} a\cdot e$, and hence
    if $a\cdot e \leq b\cdot e$ for all $e\in E$ we have
    $a = \bigovee_{e\in E} a\cdot e \leq \bigovee_{e\in E} b\cdot e = b$, which proves that it is indeed order reflecting.

    So let us prove that $\bigovee E = 1$. Suppose~$u$ is an upper bound for~$E$;
    we must show that~$u=1$.
    Note that~$\floor{u}$ is an upper bound for~$E$ too,
    since~$E$ contains only idempotents.
    It follows that the idempotent~$\floor{u}^\perp = \ceil{u^\perp}$ 
    is orthogonal
    to all~$e\in E$,
    which is impossible (by maximality of~$E$) unless
    $\ceil{u^\perp}= 0$.
    Hence~$\ceil{u^\perp}=0$,
    and thus~$u^\perp = 0$. 
\end{proof}

\begin{theorem}\label{thm:omegacompleteembed}
    Let $M$ be an $\omega$-complete effect monoid. Then there exist
    $\omega$-complete effect monoids $M_1$ and $M_2$ where
    $M_1$ is convex, and $M_2$ is an $\omega$-complete Boolean algebra
    such that $M$ embeds into $M_1\oplus M_2$.
\end{theorem}
\begin{proof}
    Let $E=H\cup B$ be a maximal collection of non-zero orthogonal idempotents of
    Proposition~\ref{prop:maxcollectionbooleanhalveable} such that the idempotents
    $p\in H$ are halvable, while the $q\in B$ are Boolean.

    Let $M_1 \equiv \bigoplus_{p\in H} pM$ and $M_2 \equiv \bigoplus_{q\in B} qM$.
    It is easy to see that $M_1$ is then again halvable and $M_2$ is Boolean.
    By Propositions~\ref{prop:dircompleteisconvex} and~\ref{prop:sharpeffectmonoidisbooleanalgebra} $M_1$ is convex while $M_2$ is an $\omega$-complete Boolean algebra.
    By the previous proposition $M$ embeds into 
    $\bigoplus_{p\in E}pM \cong M_1\oplus M_2$.
\end{proof}

One might be tempted to think that the above result could be strengthened to an isomorphism. The following example shows that this is not the case:

\begin{example}
    Let $X_1$ and $X_2$ be uncountably infinite sets, and let $A$ be the set of all pairs of functions 
    $$A \equiv \{(f_1:X_1\rightarrow [0,1], f_2: X_2\rightarrow \{0,1\})\}.$$
    Let $S_0, S_1\subseteq A$ be subsets where both functions are unequal to 0 respectively 1 only at a countable number of spots:
    \begin{align*}
    S_0 &\equiv \Bigl\{(f_1,f_2)~;~ \text{both\,} 
                \left[ \begin{aligned}
        \{x_1\in X_1~;~f_1(x_1)\neq 0\}\\
        \{x_2\in X_2~;~f_2(x_2)\neq 0\} 
                \end{aligned}\right.\text{\,countable}\Bigr\} \\
        S_1 &\equiv \Bigl\{(f_1,f_2)~;~ \text{both\,}
        \left[ \begin{aligned}
        \{x_1\in X_1~;~f_1(x_1)\neq 1\} \\
        \{x_2\in X_2~;~f_2(x_2)\neq 1\}\end{aligned} \right. \text{\,countable}\Bigr\}
    \end{align*}
    Finally, define $M = S_0\cup S_1$. It is then straightforward to check that $M$ is a $\omega$-complete effect monoid. It is easy to see that $M$ has no maximal halvable idempotent, and hence for any $M_1$ halvable and $M_2$ Boolean, necessarily ${M\neq M_1\oplus M_2}$.
\end{example}

Though the embedding is not an isomorphism for $\omega$-complete effect monoids, when we assume full 
directed-completeness, we can derive a stronger result.

\begin{lemma}\label{lem:maximalsummableelement}
    Let $M$ be a directed-complete effect monoid.
    Then there is a
    maximal element that is summable with itself. That is:
    there is an~$a\in M$ such that $a\perp a$ and for any other $b\geq a$ such that $b\perp b$, we have $b=a$.
    Furthermore,
 \begin{enumerate}
    \item
    $a\ovee a$ is an idempotent and
    \item
        if $s\leq (a\ovee a)^\perp$,
        then $s$ is idempotent for any~$s \in M$.
 \end{enumerate}
\end{lemma}
\begin{proof}
    Write $A \equiv \{s\in M\ ;\ s\perp s \}$.
    We will show that~$A$ has a maximal element using Zorn's Lemma.
    To this end, suppose~$D \subseteq A$ is a chain.
    We have to show that it has an upper bound in~$A$.
    If~$D$ is empty, then~$0 \in A$ is clearly an upper bound,
        so we may assume that~$D$ is not empty.
    Define~$s \equiv \bigvee D$.
    It is sufficient to show~$s \perp s$ as then~$s \in A$.

    Assume~$d, d' \in D$.
    We claim~$d \perp d'$.
    Indeed, assuming without loss of generality
    that~$d' \leq d$, we see~$d' \leq d \leq d^\perp$ 
        and so~$d \perp d'$.

    So~$\bigvee_{d' \in D} d \ovee d'$ exists.
As addition preserves suprema, we
    have~$\bigvee_{d' \in D} d \ovee d'
            = d \ovee \bigvee_{d' \in D} {d'} = d \ovee s$.
    Hence~$\bigvee_{d \in D} d \ovee s$ exists
        and~$\bigvee_{d\in D} d \ovee s
        = \bigl(\bigvee_{d \in D} d\bigr) \ovee s
        = s \ovee s$, so indeed~$s \in A$.
    By Zorn's Lemma we know there is a maximal element of~$A$.
    Pick such a maximal element~$a \in A$.

    Next we will show that~$a\ovee a$ is idempotent.
    Define~$b\equiv (a\ovee a)^\perp$.
    Note that $b\cdot b^\perp$ is summable with
    itself. As a result $b = b \cdot 1
    \geq b \cdot(b\cdot b^\perp \ovee  b\cdot b^\perp) =
    b^2 \cdot b^\perp \ovee  b^2 \cdot b^\perp$.
    Now note that $1=a\ovee a \ovee
    (a\ovee a)^\perp = a\ovee a \ovee  b \geq a\ovee a \ovee
    b^2 \cdot b^\perp \ovee  b^2 \cdot b^\perp$,
    and hence $a\ovee b^2 \cdot b^\perp$ is
    summable with itself. Since $a$ is maximal with this property
    we must have $b^2 \cdot b^\perp = 0$. But then $(b\cdot b^\perp)^2 = b^2\cdot b^\perp
    b^\perp = 0$ and so~$b\cdot b^\perp = 0$
    by Lemma~\ref{lem:nonilpotents}.
    Hence~$b^\perp = a\ovee a$ is indeed idempotent.

    Now let $s\leq (a \ovee a)^\perp\equiv  b$.
    We have to show that~$s$ is idempotent. By
    Lemma~\ref{lem:selfsummable} $s\cdot s^\perp$ is summable with itself
    and so~ $
    s\cdot s^\perp\ovee s\cdot s^\perp \leq  b 
        \equiv (a \ovee a)^\perp $ by
    Lemma~\ref{lem:summableunderidempotent}.
    Thus~$a \ovee a  \perp s\cdot s^\perp \ovee s\cdot s^\perp$
        and so~$a \ovee s \cdot s^\perp$ is summable with itself.
        By maximality of~$a$ we must have~$s \cdot s^\perp = 0$,
            which shows that~$s$ is indeed an idempotent.
\end{proof}

\begin{theorem}\label{thm:directedcompleteisconvex}
    Let $M$ be a directed-complete effect monoid.
    Then there exists a convex directed-complete effect monoid $M_1$,
    and a complete Boolean algebra $M_2$ such that
    $M\cong M_1\oplus M_2$.
\end{theorem}
\begin{proof}
    Let $a$ be a maximal element that is summable to itself from
    Lemma~\ref{lem:maximalsummableelement}. It was shown
    that~$p = (a\ovee a)^\perp$ is an idempotent such
    that any~$s\leq p$ is also an
    idempotent. Hence by an adaptation of Proposition~\ref{prop:sharpeffectmonoidisbooleanalgebra} to directed-complete effect monoids, $pM$ is a complete Boolean algebra.

    But of course $p^\perp M$ is halvable and hence by 
    Proposition~\ref{prop:dircompleteisconvex} $p^\perp M$ is convex.
    Hence letting $M=p^\perp M \oplus pM$ gives the desired result.
\end{proof}




\section{Order unit spaces}\label{sec:OUS}

In this section we will see that a convex $\omega$-complete effect monoid $M$ is isomorphic to the set of continuous functions of some basically disconnected compact Hausdorff space $X$ to the real unit interval, or equivalently by Gel'fand duality, isomorphic to the unit interval of an $\omega$-complete commutative C$^*$-algebra. We do this by using a known representation of convex effect algebras as unit intervals of order unit spaces, and then apply Yosida's representation theorem for lattice-ordered vector spaces.

\begin{definition}
    An \Define{order unit space} (OUS)
        is an ordered vector space
            together with distinguished element~$1 \in V$,
            such that for every~$v \in V$,
            there is a~$n \in \N$
            such that~$ -n\cdot 1 \leq v \leq n \cdot 1$.
An OUS is called \Define{Archimedean} provided
    that~$v \leq \frac{1}{n}\cdot 1$
    for all~$n \in \N$,
    implies~$v \leq 0$.
\end{definition}

\begin{definition}
Given an OUS~$V$, its \Define{unit interval}~$
    [0,1]_V \equiv \{v\in V~;~ 0\leq v\leq 1\}$
    is a convex effect algebra
        with~$a \perp b$ iff~$a+b \leq 1$
        and then~$a \ovee b = a+b$;
        $a^\perp = 1-a$; and the convex structure being given by the obvious
        scalar multiplication.
    We say~$V$ is \Define{$\omega$-complete/directed complete}
        whenever~$[0,1]_V$ is.
\end{definition}

\begin{lemma}\label{lem:completenessousemod}
    Let $V$ be an order unit space. It is directed complete
    if and only if it is \Define{bounded directed complete};
    that is if every
    directed subset of~$V$ with upper bound
    has a supremum.
    Similarly~$V$ is~$\omega$-complete iff it is bounded~$\omega$-complete.
\end{lemma}
\begin{proof}
Since any subset of the unit interval is obviously bounded, any
    bounded $\omega$-complete/directed-complete order unit space
    is $\omega$-complete/directed complete.
For the other direction, assume~$S \subseteq V$
    is some (countable) directed subset with upper bound~$b \in V$.
    Pick any~$a \in S$ and define~$S' \equiv \{v;\ v \in S; a \leq v\}$.
    It is sufficient to show that~$S'$ has a supremum.
    Pick any~$n \in \N$, $n\neq 0$ with~$-n \cdot 1 \leq a,b  \leq n\cdot 1$.
Then clearly~$\{\frac{1}{n} (s-a); s\in S'\} \subseteq [0,1]_V$
    has a supremum and so does~$S'$ as~$v \mapsto nv + a$
    is an order isomorphism.
\end{proof}

\begin{lemma}\label{boundeddircomplisarchous}
    A bounded $\omega$-complete OUS is Archimedean.
\end{lemma}
\begin{proof}
Assume~$V$ is a bounded~$\omega$-complete OUS.
    Let~$v \in V$ be given with~$v \leq \frac1n 1$ for all~$n \in \N$.
As~$a \mapsto -a$ is an order anti-isomorphism,
    all bounded directed subsets of~$V$ have an infimum too,
    so~$v \leq \bigwedge_n \frac1n 1 = (\inf_n \frac{1}{n})1 = 0$
    as desired.
\end{proof}

Recall that the unit interval of an OUS is a convex effect algebra.
In fact, every convex effect algebra is of this form.
\begin{theorem}\label{prop:convextotalisation}
    Let $M$ be a convex effect algebra. Then there exists an order
    unit space $V$ such that $M\cong [0,1]_V$~\cite{gudder1998representation}.
\end{theorem}

\begin{proposition}\label{prop:multiplicationisbilinear}
    Let $M$ be an $\omega$-complete convex effect monoid.
    Then the multiplication is ``bilinear'':
    $\lambda(a\cdot b) = (\lambda a)\cdot b = a\cdot (\lambda
    b)$ for any~$a,b\in M$ and~$\lambda \in [0,1]$.
\end{proposition}
\begin{proof}
    Clearly~$n(a\cdot (\frac1n b)) = a\cdot b = n \frac1n (a\cdot b)$
        and so~$a\cdot (\frac1n b) = \frac1n (a \cdot b)$.
    Hence~$m (a \cdot (\frac 1n b )) = a\cdot (\frac mn b)
            = \frac mn (a \cdot b)$ for any~$m \leq n$.
    We have shown the second equality for rational~$\lambda$.
    With a similar argument one shows the first equality holds
        as well for rational~$\lambda$.
    To prove the general case, 
        let~$a,b \in M$ and~$\lambda \in [0,1]$ be given.
    Pick a sequence~$0 \leq q_1 < q_2 < \ldots $ of rationals
        with~$\sup_n q_n = \lambda$.
    Let~$V$ be an OUS with~$M \cong [0,1]_V$ (as a convex effect algebra),
        which exists due to Theorem~\ref{prop:convextotalisation}.
    Note that the multiplication of~$M$ is only defined on~$[0,1]_V$.
    For~$a \in[0,1]_V$ we have~$a \leq \| a \| 1$
        and so~$\| a \cdot b \| \leq \| a \| \| b \|$ for any~$b \in [0,1]_V$.
        Hence
        \begin{align*}
            &\norm{\lambda (a\cdot b) - a\cdot (\lambda b)} \\
            &\quad \ =\  \norm{(\lambda - q_i)(a\cdot b) + a\cdot (q_i b) - a\cdot (\lambda b)} \\
        &\quad\ =\  \norm{(\lambda-q_i)(a\cdot b) - a\cdot ((\lambda - q_i)b)} \\
        &\quad\ \leq\  (\lambda-q_i)\norm{a\cdot b} + (\lambda-q_i)\norm{a}\norm{b}.
        \end{align*}
        The right-hand side vanishes
            as $i\rightarrow \infty$.
            Thus $\norm{\lambda(a\cdot b) - {a\cdot (\lambda b)}=0}$.
            Since $V$ is Archimedean by Lemma \ref{boundeddircomplisarchous}, $\norm{\cdot}$
            is a proper norm so that then $\lambda(a\cdot b) - a\cdot (\lambda b) = 0$. That $\lambda(a\cdot b) = (\lambda a)\cdot b$ follows similarly.
\end{proof}

\begin{theorem}[{Yosida \cite{yosida}, cf.~\cite{westerbaan2016yosida}}]\label{prop:yosida}
Let~$V$ be a norm-complete lattice-ordered Archimedean OUS. Denote by $\Phi$ the \Define{Yosida spectrum} of~$V$:
        the compact Hausdorff space
        of real-valued finite-suprema-preserving unital linear functionals on~$V$
        with the induced pointwise topology of~$\R^V$.
    Then the map $\vartheta \colon V \rightarrow C(\Phi)$
    given by~$\vartheta(v)(\varphi) = \varphi(v)$
    is an order isomorphism.
\end{theorem}
\begin{remark}
    In fact the category of compact Hausdorff spaces with continuous maps between
        them is dually equivalent to the category of lattice-ordered norm-complete Archimedean
            order unit spaces with linear unital finite-supremum-preserving
            maps between them.~\cite{westerbaan2016yosida}
\end{remark}


\begin{theorem}\label{thm:convexextremallydisconnected}
    Let $M$ be a convex $\omega$-complete effect monoid.
    Then $M \cong [0,1]_{C(X)}$ for some basically
        disconnected Hausdorff space~$X$.
    If~$M$ is even directed complete, then~$X$ is extremally disconnected.
\end{theorem}
\begin{proof}
    The effect monoid~$M$ is a lattice by Theorem~\ref{thm:latticeemon}.
    By Theorem~\ref{prop:convextotalisation}
            there is an OUS~$V$
            such that~$M \cong [0,1]_V$ as a convex effect algebra.
        It is easy to see that~$V$ is a lattice as well
            as any supremum reduces to one in the interval~$[0,1]_V$.
    $V$ is bounded~$\omega$-complete
        by Lemma~\ref{lem:completenessousemod} and hence Archimedean by Lemma~\ref{boundeddircomplisarchous},
        and norm complete by~\cite[Lemma 1.2]{wright1972measures}.
Then by Theorem~\ref{prop:yosida}
    we have~$V \cong C(\Phi)$,
        where~$\Phi$ is the Yosida spectrum of~$V$.
    Note~$C(\Phi)$ is bounded~$\omega$-complete
     iff~$\Phi$ is basically disconnected
     and~$C(\Phi)$ is bounded directed complete
        iff~$\Phi$ is extremally disconnected.
The multiplication of~$M$
    induces a multiplication on~$[0,1]_{C(\Phi)}$,
    which we will denote by~$f * g$ to distinguish it
    from the pointwise multiplication~$f \cdot g$.
To prove the Theorem,
    it remains to be shown that~$f * g = f \cdot g$.

Pick any~$\varphi \in \Phi$.
It is sufficient to show that~$\varphi(f * g) = \varphi(f) \varphi(g)$
    --- indeed then~$(f * g)(\varphi) = \varphi( f * g) = \varphi(f) \varphi(g)
                = f(\varphi) g(\varphi) = (f\cdot g)(\varphi)$.
The remainder of the proof is based on~\cite[Lemma~5.26]{riesz}.
Write~$d(a,b) \equiv a\vee b - a\wedge b$
        for either~$a,b \in [0,1]$ or~$a,b \in [0,1]_{C(\Phi)}$.
    Note~$f * (g \vee h) \leq (f*g) \vee (f*h)$
    and~$f * (g \wedge h) \geq (f*g) \wedge (f*h)$,
    hence~$d(f * g, f*h) \leq f * d(g,h)$ for~$f,g,h \in [0,1]_{C(\Phi)}$.
So in particular~$d(f * g, f * (\varphi(g) 1)) \leq f * d(g, \varphi(g) 1)
                    \leq d(g, \varphi(g)1)$.
The multiplication~$*$ is ``bilinear''
     by Proposition~\ref{prop:multiplicationisbilinear},
        so~$f * (\varphi(g) 1) = \varphi(g) f$.
By definition~$\varphi$ preserves infima, suprema and~$1$
    and so~$d(\varphi(f), \varphi(g)) = \varphi(d(f,g))$
        for any~$f,g \in [0,1]_{C(\Phi)}$.
Combining the last three facts, we see
\begin{align*}
        d(\,\varphi(f*g)\,,\, \varphi(f)\varphi(g)\,)
                & \ =\  \varphi(d(\, f*g\,,\, f * (\varphi(g) 1)\,)) \\
               & \ \leq\  \varphi(d(g, \varphi(g)1)) \\
               & \ =\  d(\varphi(g), \varphi(g)) \ = \ 0.
\end{align*}
    So indeed~$\varphi(f*g) = \varphi(f)\varphi(g)$,
        hence~$f*g = f\cdot g$.
\end{proof}

\begin{remark}%
The previous theorem can also be proven
    using Kadison's representation Theorem~\cite{kadison1951representation},
    which states that any norm-complete Archimedean OUS
    with a bilinear and positive product
    is isomorphic to~$C(X)$.
This requires one to show that the product on~$[0,1]_V$
    extends to a bilinear and positive product on the entirety of $V$,
    which can be done,
    but is a bit tedious (cf.~\cite[Theorem 46]{basmaster}).
\end{remark}

\section{Main theorems}\label{sec:maintheorems}

We now collect our previous results and prove our main theorems.

\begin{theorem}\label{thm:omega-complete-classification}
    Let $M$ be an $\omega$-complete effect monoid.
    Then there exists a basically disconnected compact Hausdorff space $X$,
    and an $\omega$-complete Boolean algebra $B$ such that 
    $M$ embeds into $[0,1]_{C(X)} \oplus B$.
\end{theorem}
\begin{proof}
    By Theorem~\ref{thm:omegacompleteembed} there exist
    $\omega$-complete effect monoids $M_1$ and $M_2$
    such that $M$ embeds into $M_1\oplus M_2$, where $M_1$ is convex and 
    $M_2$ is an $\omega$-complete Boolean algebra.
    By Theorem~\ref{thm:convexextremallydisconnected} $M_1=[0,1]_{C(X)}$
    for a basically disconnected compact Hausdorff space $X$.
\end{proof}

\begin{theorem}\label{mainthmdirectedcomplete}
    Let $M$ be a directed-complete effect monoid.
    Then there exists an extremally-disconnected compact Hausdorff space $X$,
    and an complete Boolean algebra $B$ such that 
    $M\cong [0,1]_{C(X)} \oplus B$.
\end{theorem}
\begin{proof}
    Same as previous theorem but using Theorem~\ref{thm:directedcompleteisconvex}.
\end{proof}

\begin{corollary}
	Let $M$ be an $\omega$-complete effect monoid. Then $M$ is commutative.
\end{corollary}
\begin{proof}
	$M$ embeds into $[0,1]_{C(X)}\oplus B$ where $B$ is a Boolean algebra and $X$ is a basically disconnected compact Hausdorff space. The multiplication in $B$ is given by the join and hence is commutative, while the multiplication in $C(X)$ is given by pointwise multiplication in the real numbers and hence is also commutative. The multiplication of $M$ is then necessarily also commutative.
\end{proof}

\begin{theorem}\label{thm:no-zero-divisors}
    Let $M$ be an $\omega$-complete effect monoid with no non-trivial
    zero divisors.
    Then either $M=\{0\}$, $M=\{0,1\}$ or $M\cong [0,1]$.
\end{theorem}
\begin{proof}
    Assume that $M\neq \{0,1\}$ and $M\neq \{0\}$. 
    We remark first that for any idempotent $p\in M$ we have $p\cdot p^\perp = 0$,
    and hence by the lack of non-trivial zero divisors we must have $p=0$ or $p=1$.
    Since $M\neq \{0,1\}$, there is an $s\in M$ such that $s\neq 0,1$, and hence
    we must have $s\cdot s^\perp \neq 0$. 
    By Lemma~\ref{lem:selfsummable} we then
    have an element $2(s\cdot s^\perp)$ that is halvable.
    Hence by Lemma~\ref{lem:ceilhalveable} 
    $\ceil{2s\cdot s^\perp}$ is also halvable.
    As this ceiling is an idempotent it must be equal to $1$ or to $0$. 
    If it were zero then 
    $2s\cdot s^\perp \leq \ceil{2s\cdot s^\perp} = 0$, 
    which contradicts $s\cdot s^\perp \neq 0$.
    So $1=\ceil{2s\cdot s^\perp}$ is halvable. 
    By Proposition~\ref{prop:dircompleteisconvex}, $M$ is then convex.
    Hence, by Theorem~\ref{thm:convexextremallydisconnected}~$M = [0,1]_{C(X)}$
        for some basically disconnected $X$.
    We will show that $X$ has a single element, which will complete the proof.

    As idempotents of~$[0,1]_{C(X)}$ correspond to clopens of~$X$,
        there are only two clopens in~$X$, namely~$X$ and~$\emptyset$.
    Reasoning towards contradiction,
        assume there are~$x,y \in X$ with~$x \neq y$.
    By Urysohn's lemma we can find~$f \in C(X)$
        with~$0 \leq f \leq 1$, $f(x) = 0$ and~$f(y) = 1$.
        $U_x \equiv f^{-1}([0,\frac13)$ and~$U_y\equiv f^{-1}((\frac23,1])$
        are two open sets with disjoint closure.
    Using Urysohn's lemma again, we can find~$g \in C(X)$
        with~$g(\overline{U_x}) = \{0\}$ and $g(\overline{U_y}) = \{1\}$.
    As~$X$ is basically disconnected,
        we know~$\overline{\supp g}$ is clopen.
    We cannot have~$\overline{\supp g} = \emptyset$
        as~$y \in U_y \subseteq \overline{\supp g}$.
        Hence~$\overline{\supp g} = X$.
        But then~$x \in U_x \subseteq X \backslash \overline{\supp g} = \emptyset$.
        Contradiction.
        Apparently~$X$ has only one point and so~$M \cong [0,1]$.
\end{proof}

\begin{remark}
In~\cite{kentathesis} it is shown that any $\omega$-effectus
    satisfying a natural property called \emph{normalisation} (basically,
    that any substate can be normalised to a  state) must have scalars
    without zero divisors. We have hence also completely classified the
    allowable scalars in a $\omega$-effectus satisfying normalisation.
\end{remark}

\begin{remark}
    As noted in Example~\ref{ex:CX}, the unit interval of any partially ordered ring where the sum and product of positive elements is again positive forms an effect monoid. Hence the previous results can also be seen as, to our knowledge, new results for $\sigma$-Dedekind-complete ordered rings. Although it is not clear how our results for the unit interval can be extended to results concerning the entire ring, our results might eventually lead to a version of the characterisation of the real numbers as the unique Dedekind-complete ordered field, where the condition of being a field can be weakened to being a domain (i.e. a ring without non-trivial zero-divisors).
\end{remark}

\section{Conclusion and outlook}\label{sec:conclusion}
We have shown that any $\omega$-complete effect monoid 
	embeds into a direct sum of a Boolean algebra and 
	the set of continuous functions from a given 
	Hausdorff space $X$ into the real unit interval $[0,1]$. 
As a result, the only $\omega$-complete effect monoids 
	without any zero divisors are either degenerate ($\{0\}$), 
	the Booleans ($\{0,1\}$), or the real unit interval ($[0,1]$),
	resulting in a dichotomy of sharp logic 
	and fuzzy probabilistic logic. 

The structure of an $\omega$-complete effect monoid; 
	partial addition, involution, multiplication 
	and directed limits; 
	captures in a sense the structure present 
	in the real unit interval and hence these results 
	give a foundational underpinning to why 
	the unit interval should be the designated structure 
	of the scalars in any non-sharp physical or logical theory.

An interesting follow-up question is to consider whether the $\omega$-completeness can be weakened somehow. There exist pathological (non-commutative)
    effect monoids
    (\cite[Ex.~4.3.9]{kentathesis} and~\cite[Cor.~51]{basmaster}), but all the known ones
    are non-Archimedean (i.e.~they have infinitesimal elements).
    So perhaps it is possible to embed any Archimedean effect monoid
    into a directed-complete effect monoid, similar
    to the Dedekind--MacNeille completion of a poset into a
    complete lattice.
    If this is the case, then we can essentially get rid of the assumption of $\omega$-completeness.

Another matter that deserves further investigation
    is the categorical aspect of our results.
It can be shown that Theorem~\ref{mainthmdirectedcomplete}
    has the following consequence:
    the category of directed-complete effect monoids
    with effect monoid homomorphisms 
    between them (which need not preserve suprema)
    is dually equivalent to the category
        which has as objects pairs~$(X,S)$,
            where~$S \subseteq X$ is a clopen subset
            of a extremally disconnected compact Hausdorff space~$X$,
    and morphisms~$(X,S) \to (Y,T)$
           are continuous maps~$f\colon X \to Y$
            satisfying~$f(S) \subseteq T$.
Obtaining a similar duality for~$\omega$-complete
effect monoids
    would in our estimation be rather non-trivial, but as it involves $\omega$-complete algebras, it can perhaps use results analogous to those of Ref.~\cite{furber2017unrestricted}.
        
We also aim to use our results regarding 
	$\omega$-complete effect monoids to study 
	$\omega$-effectuses~\cite{kentathesis}.
It seems possible to use this framework in combination 
	with our results to make a 
	``reconstruction of quantum theory''~\cite{hardy2001quantum,chiribella2011informational} 
	that does not a priori refer to the structure 
	of the real unit interval,
	\footnote{In the reconstruction of	
	Ref.~\cite{tull2016reconstruction}, almost all the work 
	is done without any assumptions on the scalars, 
	but in the end, the structure of the unit interval 
	is necessary to get to standard quantum theory, 
	which is something which would not be necessary 
	with our approach.} 
	for instance by adopting the effectus-based axioms 
	of Ref.~\cite{wetering2018reconstruction}.

Finally, when the associative bilinear multiplication of 
    an effect monoid is replaced by a binary operation 
    satisfying axioms related to the L\"uders product
    $(a,b)\mapsto \sqrt{a}b\sqrt{a}$ 
    for positive $a$ and $b$ in a C$^*$-algebra, 
    one gets a \emph{sequential effect algebra}. 
In forthcoming work 
~\cite{second} 
we use the techniques 
    and results of this paper
    to show that a spectral theorem holds in any
    \emph{normal} sequential effect algebra~\cite{gudder2002sequential}
    and furthermore that any such algebra is isomorphic 
    to the direct sum of a
    Boolean, convex and a ``almost convex'' 
    normal sequential effect algebra.

\noindent \textbf{Acknowledgements}: \censor{JvdW} is supported in part by \censor{AFOSR} grant \censor{FA2386-18-1-4028}.

\bibliographystyle{IEEEtran}
\bibliography{main}

\end{document}